\documentclass[12pt]{amsart}

\usepackage[
text={440pt,575pt},
headheight=9pt,
centering
]{geometry}

\usepackage{amsmath, amssymb, amsthm, enumerate, bm}
\usepackage{mathtools}
\usepackage{dsfont}

\allowdisplaybreaks 

\usepackage{mathptmx}

\usepackage{caption}
\usepackage{tabularx}
\newcolumntype{x}[1]{!{\centering\arraybackslash\vrule width #1}}
\usepackage{booktabs}
\usepackage{xcolor}
\usepackage{hyperref}
\definecolor{darkblue}{RGB}{0,0,160}
\hypersetup{
colorlinks,%
citecolor=darkblue,%
filecolor=black,%
linkcolor=darkblue,%
urlcolor=darkblue
}

\usepackage{todonotes}
\newcolumntype{E}{>{\footnotesize \selectfont}l<{}} 
\newcolumntype{F}{>{\small\selectfont}l<{}}

\newcommand{\C}{\mathbb{C}}

\newcommand{\N}{\mathbb{N}}

\newcommand{\R}{\mathbb{R}}

\newcommand{\V}{\mathbb{V}}

\newcommand{\Nc}{\mathcal{N}}

\DeclareMathOperator{\diag}{diag}

\DeclareMathOperator{\pnt}{\raise 0.5mm \hbox{\large\bf.}}

\newtheoremstyle{thm}{}{}
     {\em}
     {}
     {\bf}
     {.}
     {0.5em}
     {\thmname{#1}\thmnumber{ #2}\thmnote{ #3}}

\newtheoremstyle{def}{}{}
     {\rm}
     {}
     {\bf}
     {.}
     {0.5em}
     {\thmname{#1}\thmnumber{ #2}\thmnote{ #3}}

\theoremstyle{thm}

\newtheorem{thm}{Theorem}[section]
\newtheorem{lem}[thm]{Lemma}
\newtheorem{cor}[thm]{Corollary}
\newtheorem{prop}[thm]{Proposition}
\newtheorem{conj}[thm]{Conjecture}

\theoremstyle{def}

\newtheorem{rem}[thm]{Remark}
\newtheorem{exa}[thm]{Example}
\newtheorem{Def}[thm]{Definition}
\newtheorem{Alg}[thm]{Algorithm}
\newtheorem{Not}[thm]{Notation}

\let\epsilon=\varepsilon
\let\phi=\varphi
\let\kappa=\varkappa

\makeatletter
\@namedef{subjclassname@2020}{%
  \textup{2020} Mathematics Subject Classification}
\makeatother

\title[Covariance matrices of length power functionals
of random geometric graphs]{
Covariance matrices of length power functionals
of random geometric graphs -- an asymptotic analysis
}

\author{Matthias Reitzner, Tim R\"omer, and Mandala von Westenholz}

\subjclass[2020]{Primary: 15B52, 60D05; Secondary: 05C80, 60G55}

\keywords{covariance matrix, length power functional, Poisson point process}

\begin{document}
\begin{abstract}
Asymptotic properties of a vector of length power functionals of random geometric graphs are investigated. More precisely, its asymptotic covariance matrix is studied as the intensity of the underlying homogeneous Poisson point process increases. This includes a systematic discussion of matrix properties like rank, definiteness, determinant, eigenspaces or decompositions of interest. 
For the formulation of the results a case distinction is necessary. Indeed, in the three possible regimes the respective covariance matrix is of quite different nature 
which leads to different statements. Finally, stochastic  consequences for random geometric graphs are derived.
\end{abstract}

\maketitle

\section{Introduction} 
\label{Kap: Regimes}

Studying data seen as point clouds is of great interest in data analysis, machine learning, and related fields. One very interesting recent approach is topological data analysis where geometric as well as topological information is investigated for applications like the use of persistent homology or visualization methods (see, e.g., \cite{Carlsson}, \cite{CarlssonZomorodian}, \cite{EdelsbrunnerHarer}, \cite{EdelsbrunnerLetscher}, \cite{Wasserman}). In particular, in the aforementioned areas, simplicial complexes and their properties are of significance like Vietoris-Rips as well as C\v{e}ch complexes (see \cite{EdelsbrunnerHarer}). Their common $1$-skeleton is a geometric graph and is one of the main objects in the following studies. 

More precisely, given a point cloud $X\subseteq \R^d$ and a given value $\delta$, the geometric graph $\mathcal{G}(X, \delta)$ is defined as follows. Vertices are the points from $X$ and there exists an edge between two such points if and only if the distance between them is smaller or equal to $\delta$.

A central idea of stochastic geometry to study generic behaviours of objects of interest is the use of point processes to construct random ones. In particular, this yields  random simplicial complexes and geometric graphs (see, e.g., \cite{Adler}, \cite{AkinwandeReitzner} \cite{BobrowskiKahle}, \cite{Grygiereketal2020}, \cite{Kahle}, \cite{Penrose}) , where the vertices originate from a \emph{Poisson point process} $\eta$ (see \cite{LastPenrose}). Such graphs have been studied from various points of views (see, e.g., \cite{Reitzner}). Note that here the original approach is due to Gilbert for modelling networks \cite{Gilbert}. 

In this paper we consider a sequence of homogeneous Poisson point processes $\eta_t$ with \emph{intensity} $t$. 
It is always assumed that the processes are supported in a convex compact set $W \subseteq \R^d$. 
Given the processes $\eta_t$ and a sequence of distances $\delta_t$, we study asymptotic behaviours
of the induced random graphs $\mathcal{G}(\eta_t, \delta_t)$ as $\delta_t \rightarrow 0$ for $t \rightarrow \infty$. 

Recall, that in the random geometric graph $\mathcal{G}(\eta_t, \delta_t)$ there exists an edge between vertices $x, y \in \eta_t$ with $x\neq y$ if and only if  
$
\parallel x -y \parallel \leq \delta_t.
$
In the past, Poisson functionals have been investigated at various places  (see, e.g., \cite{Peccati1}, \cite{Peccati} or \cite{Yoshifusa}). For example, research work about a length power functional $L_t^{(\tau)}$ of a random geometric graph has been carried out. In particular, Reitzner--Schulte--Thäle \cite{Reitzner}  studied certain algebraic properties of the \emph{asymptotic covariance matrix} $\Sigma_n$ of a vector $(\tilde{L}_t^{(\tau_1)}, \ldots,\, \tilde{L}_t^{(\tau_n)})$ of  normalized and centered length power functionals. Taking especially this work as a motivation the main goals of this paper are:

\begin{enumerate} 
\item Expanding the analysis of $\Sigma_n$ as extensive as possible concerning algebraic aspects like rank, definiteness, determinant, eigenspaces and selected decompositions.
\item Using the gained algebraic results in (i) for a stochastic analysis on the entry wise length power functionals in the considered vector. 
\end{enumerate}
 
In dependence to the asymptotic behavior of the term $t\delta_t^d$, which occurs as the expected degree of a vertex in the random geometric graph, the following regimes arise:
\begin{itemize}
\item \emph{subcritical (sparse) regime}  for $t\delta_t^d \xrightarrow{t \rightarrow \infty} 0$,
\item \emph{critical (thermodynamic) regime}  for $t\delta_t^d \xrightarrow{t \rightarrow \infty} c \in \R_{> 0}$,
\item \emph{supercritical (dense) regime} for $t\delta_t^d \xrightarrow{t \rightarrow \infty}  \infty$.
\end{itemize} 

The structure of the covariance matrix $\Sigma_n$ differs significantly in each regime. This is introduced in Section \ref{Kap: Poisson-Funktionale}. Remarkably, the matrix in the critical regime can be decomposed into the corresponding matrices of the other two regimes. 
For the convenience of the reader we summarize in Section \ref{Kap: summary} all main results of the manuscript. The asymptotic analysis needs tools related to symmetric positive semidefinite matrices, which are discussed in Section \ref{Kap: Positive semidefinite symmetric matrices}. 
The core part of our investigations is started 
right afterwards. Due to the different matrix structures in the regimes, there are separate examinations in Sections \ref{Kap: Superkrit. Reg.}--\ref{Kap: Krit. Reg.}. Every question of interest is answered in the supercritical regime. 
Beside substantial progress, in the other regimes there remain also open problems. Cross-regime observations related to decompositions are discussed in Section \ref{Kap. Regimeunabhaengige Untersuchung II}. 
In Section \ref{Kap: stoch. appl.} we apply our results to the original stochastic situation. It is not too difficult to show 
(see Theorem \ref{thm:stochconv-dense}) that in the supercritical regime
\[
\tilde L_t^{(\tau)}  - \frac d{\tau+d} \tilde L_t^{(0)} \stackrel{\mathbb{P}}{\to} 0
\text{ as } t \to \infty.
\]
Thus, asymptotically the number of edges $ \tilde L_t^{(0)} $ determines all other length power functionals. More involved are the investigations in the
other regimes. In Theorem \ref{thm:stochconv-(sub)crit} we prove that
\[
\tilde L_t^{(\tau)}  - \frac d{\tau+d} \tilde L_t^{(0)} \stackrel{\mathbb{P}}{\to} \frac{1}{\max\{\sqrt c,1\}}Z
\text{ as } 
t \to \infty,
\]
where $Z$ is a normal distributed random variable independent of $\tilde L_t^{(0)}$. Hence, the length power functionals asymptotically decompose into the number of edges and some independent noise.
Finally, in Section \ref{Kap: Resumee} open questions 
and possible strategies to solve them 
are discussed.

\section{Preliminaries} \label{Kap: Poisson-Funktionale}
A Poisson point process is always given as $\eta_t \colon \Omega  \rightarrow \mathbf{N}(W), \,\omega \mapsto \sum_{i=1}^{N(\omega )} \delta_{X_i(\omega)}$, where 
$W \subseteq \R^d$ is a convex compact set, 
$X_i \colon \Omega \rightarrow W$ for $i \in [N(\omega)]$
as well as $N \colon \Omega \rightarrow \N$ are random variables, and 
$\mathbf{N}(W)$ is the set of finite counting measures on $W$. 
Here we also write $[n]=\{1,\dots,n\}$ for $n\in \N$.
Throughout the paper we use: 

\begin{Not} 
\label{Not: PPP} \
\begin{enumerate}
\item The notation $\eta_t(\omega)$ means the counting measure  $\sum_{i=1}^{N(\omega)} \delta_{X_i(\omega)}$ and simultaneously also the point set $\{X_1(\omega), \ldots, \,X_{N(\omega)}(\omega)\}$ understood as the support of the measure.
\item Realizations are simply written as $\eta_t=\eta_t(\omega)$, 
$ x_i=X_i(\omega)$, and $N=N(\omega)$.
\item Given a Borel set $A\in \mathcal{B}(W)$
the abbreviation 
$\eta_t(A)=\eta_t(\omega)(A)=\sum_{i=1}^N \delta_{x_i}(A)$ is induced by (ii) and is the number of random points in $A$.
If we identify $\eta_t$ with its support, then 
$\eta_t\cap A=\eta_t(\omega)\cap A=
\{x_1,\ldots, \,x_n\}\cap A$ stands for the points lying in $A$.
\item $\eta_{t, \,\neq}^m (\omega)= \{ (x_{1}, \ldots, \,x_{m}) \mid 
x_i =X_{j_i}(\omega)  \text{ pairwise distinct}
\text{ for }
j_i\in [N(\omega)]\}$ 
is the set of all $m$-tuples of pairwise distinct points in $\eta_t$. We also write $\eta_{t, \,\neq}^m$ for $\eta_{t, \, \neq}^m(\omega)$.
\end{enumerate}  
\end{Not}

The discussion in the following is based on  \textit{Poisson functionals} in $W$. Such a functional with an underlying Poisson point process $\eta \colon \Omega \rightarrow \mathbf{N}(W)$ is a random variable
\[F \colon \Omega \rightarrow \R
\text{ with a decomposition }
F = f\circ \eta \hspace{4pt} \text{ $\mathbb{P}$-a.s}.
\] 
for a measurable function $f \colon  \mathbf{N}(W) \rightarrow \R$. For further details see, e.g., \cite{SchneiderWeil}.
In particular, applications in the context of our work are discussed
in
\cite{AkinwandeReitzner}, \cite{Grygierek}, \cite{Peccati1}, and \cite{Reitzner}. As above we also shorten $F(\omega) = f(\eta(\omega))$ by $F = f(\eta)$. The central Poisson functional of this paper is the so-called \textit{length power functional}
\begin{align}
\label{Eq: Laengen-Potenz-Funktional}
L_t^{(\tau)}=L^{(\tau)}(\mathcal{G}( \eta_t, \, \delta_t))= \frac{1}{2} \sum_{(x_1,\, x_2) \in \eta_{t,\, \neq}^2} \mathds{1}( \parallel x_1- x_2\parallel \leq \delta_t ) \parallel x_1-x_2 \parallel^{\tau}
\text{ for }  \tau \in \R,
\end{align} 
where here and in the following $\mathds{1}( \cdot)$ is an indicator function. The functional sums up the $\tau$-th powers of the edge lengths in the random geometric graph $\mathcal{G}( \eta_t, \, \delta_t)$. For $\tau = 1$ this corresponds to the total edge length and for $\tau =0$ to the number of edges in the random geometric graph. Note that this number and generalizations of it were already investigated in \cite{Grygierek} from an stochastic point of view. Our main goal is an extensive examination of the asymptotic covariance matrix of a functional vector 
\[(\tilde{L}_t^{( \tau_1)}, \ldots ,\,\tilde{L}_t^{(\tau_n)})\] 
for powers 
$-d/2<\tau_1 < \ldots< \tau_n$ of \textit{centered and normalized length power functionals}  
\begin{align} \label{Eq: normiertes, zentriertes Laengen Potenz Funktional}
\tilde{L}_t^{(\tau)} :=\frac{\left(L_t^{(\tau)} - \mathbb{E}(L_t^{(\tau)})\right)} {\max\{t\delta_t^{\tau+d/2}, \,t^{3/2}\delta_t^{\tau +d }\}}.
\end{align} 

This matrix was explicitly determined 
in \cite[Thm. 3.3]{Reitzner}  by Reitzner--Schulte--Thäle:

\begin{thm}[(\cite{Reitzner})]
\label{Satz: Kov.matrix Laengen Potenz Funktional}
Let $ -d/2<\tau_1 < \ldots < \tau_n$. The random vector 
$(\tilde{L}_t^{(\tau_1)}, \ldots , \,\tilde{L}_t^{(\tau_n)})$ has the asymptotic covariance matrix  
\begin{align}
\Sigma_n = \begin{cases}  
\Sigma_n^{\emph{sb}} &\text{ for } \, \lim\limits_{t \rightarrow \infty} t\delta_t^d =0,  \\
\Sigma_n^{\emph{sb}} + c \Sigma_n^{\emph{sp}}& \text{ for } \,\lim\limits_{t \rightarrow \infty} t\delta_t^d =c \in (0, \, 1],   \\
\frac{1}{c}\Sigma_n^{\emph{sb}}+ \Sigma_n ^{\emph{sp}} & \text{ for }\,\lim\limits_{t \rightarrow \infty} t\delta_t^d \in (1, \,\infty), \\
\Sigma_n^{\emph{sp}} &\text{ for }  \,\lim\limits_{t \rightarrow \infty} t\delta_t^d =\infty \\
\end{cases} \label{Eq: Kov.matrix Laengen Potenz Funktional} 
\end{align} 

with  
$\Sigma_n^{\emph{sb}} := 
V(W) \left(\frac{d\cdot \kappa_d}{2( \tau_i + \tau_j +d) }\right), \, \Sigma_n^{\emph{sp}} := V(W) \left(\frac{d^2\cdot \kappa_d^2}{(\tau_i +d) (\tau_j +d)}\right) \in \R^{n \times n}$
where $V(W)$ and $\kappa_d$
are the volumes of $W$  and of the $d$-dimensional unit ball (w.r.t. the Lebesgue measure).
\end{thm} 

For example, for $n=1$ one obtains $\Sigma_n=(\sigma_{11})$ 
for a certain constant $\sigma_{11}>0$
described above. All questions studied in this paper are trivial in this case. Therefore we assume $n\geq 2$ in the remaining discussion.
$\Sigma_n$ is symmetric, positive semidefinite and  hence diagonalizable. We denote its real eigenvalues sorted and in increasing order by $
\lambda_1\leq \ldots\leq \lambda_n.
$

\section{Main results} \label{Kap: summary}

The following tables present an overview of our algebraic results for the asymptotic covariance matrix $\Sigma_n$ for $n\geq 2$, which are described in each regime. Open cases are also mentioned. 
In Sections \ref{Kap: Superkrit. Reg.}--\ref{Kap: Krit. Reg.}
first properties are discussed regime-dependent
and include rank, determinant, definiteness, and eigenspaces. 
A summary is:
 
\begin{center}
\captionsetup{type=table}
\begin{tabular}[h]{| F x{1pt}E|E|E|}

\toprule    

&\small{\textbf{supercritical regime}} & \small{\textbf{subcritical regime}}  & \small{\textbf{critical regime} }\\
\toprule

\textbf{rank}&$\text{rank}(\Sigma_n)=1$ (\ref{Kor: Rang superkrit. Reg.}(i))& maximal rank (\ref{Kor: Det. Kov.matrix subkrit. Reg.}(i))&maximal rank (\ref{Satz: pos. Def., Inv., Det. krit. Reg.}(i))\\  
\midrule
 
\textbf{invertibility}& singular (\ref{Kor: Rang superkrit. Reg.})&regular (\ref{Kor: Det. Kov.matrix subkrit. Reg.}), &regular (\ref{Satz: pos. Def., Inv., Det. krit. Reg.}),\\

&&inverse in \ref{Satz: Inv. subkrit. Reg.}&inverse in \ref{Satz: Inverse krit. Regime}\\
\midrule
 
\textbf{Jordan}& diagonalizable & diagonalizable& diagonalizable \\

\textbf{normal form}&  & &\\
\midrule 
 
\textbf{determinant} &  $\text{det}(\Sigma_n)=0$ (\ref{Kor: Rang superkrit. Reg.}(ii))& $\text{det}(\Sigma_n)>0$ (\ref{Kor: Det. Kov.matrix subkrit. Reg.}),& $\text{det}(\Sigma_n)>0$ (\ref{Satz: pos. Def., Inv., Det. krit. Reg.}(ii)), \\

&& explicit formula in (\ref{Eq: Det. Sigma subkrit. Reg})&explicit formula in \ref{Satz: Det. krit. Reg.}\\ 
\midrule
 
\textbf{definiteness} & positive semidefinite& positive definite (\ref{Kor: Det. Kov.matrix subkrit. Reg.}(ii)) &positive definite (\ref{Satz: pos. Def., Inv., Det. krit. Reg.}(iii))\\ 
 
& not positive definite (\ref{Kor: Rang superkrit. Reg.}(iii))&& \\
 
\midrule

\textbf{characteristic}& explicit formula in (\ref{Kor: Rang superkrit. Reg.} (iv))&explicit formula in (\ref{Satz: char. Poly. Sigma subkrit. Reg. }) &formula in (\ref{Satz: char. Poly. Sigma krit. Reg.})\\

\textbf{polynomial}&  &&\\

\midrule

\textbf{eigenvalues}& explicit formula in (\ref{Satz: EW.e, ER.e superkrit. Reg.}) &bounds in (\ref{Satz: Schranken EWe Sigma subkrit. Reg.}),&bounds in (\ref{Satz: Schranken EWe krit. Reg.}) and\\

&& explicit formula &(\ref{Bem: Schranken EWe krit. Reg.}),    \\

&&for $n=2$ 
& 
explicit formula \textit{open} \\
\midrule    

\textbf{eigenspaces}& explicit formula in (\ref{Satz: EW.e, ER.e superkrit. Reg.}) & \textit{open }&\textit{open }\\
\bottomrule

\end{tabular}
\caption{\small{first results in all regimes}} \label{Tab: first results in all regimes}
\end{center} 

The supercritical regime is completely solved in the sense that we obtain for every aspect either explicit formulas or qualitative statements. In the other two regimes some questions are still open. We would like to point out that from an qualitative point of view
the results in the subcritical and critical regime coincide, since the critical case is just a rank-$1$-disturbance of the subcritical one. In Sections \ref{Kap: Superkrit. Reg.}--\ref{Kap: Krit. Reg.} we study also matrix decompositions and prove:

\begin{center}
\captionsetup{type=table}
\begin{tabular}[h]{| F x{1pt}E|E|E|} 

\toprule    

&\small{\textbf{supercritical regime}} & \small{\textbf{subcritical regime}}  & \small{\textbf{critical regime} }\\
\toprule

\midrule
\textbf{LU }& explicit formula in (\ref{Satz: LR-Zerleg. superkrit. Reg.}(i)) &explicit formula in (\ref{Satz: LR-Zerleg. subkrit. Reg.})&explicit formula for the \\
\textbf{decomposition}&  & &vector of natural increasing \\
\textbf{}&  & &powers in (\ref{Satz: LR-Zerlegung krit. Reg.})\\

\midrule
\textbf{Cholesky }& explicit formula in (\ref{Satz: LR-Zerleg. superkrit. Reg.}(ii)) &explicit formula in (\ref{Satz: Cholesky-Zerleg. subkrit. Reg.})&explicit formula for the \\
\textbf{decomposition}&  & &vector of natural increasing \\
\textbf{}&  & &powers in (\ref{Satz: Cholesky-Zerleg. krit. Reg.})\\

\midrule
\textbf{Square root }& explicit formula in (\ref{Satz: LR-Zerleg. superkrit. Reg.}(iii)) & \textit{open}& \textit{open} \\

\bottomrule
\end{tabular}
\caption{\small{decompositions in all regimes}}\label{Tab: decompositions in all regimes}
\end{center} 
   
Observe that LU and Cholesky decompositions are derived in the supercritical as well as in the subcritical regime. In the critical regime they are investigated only in the special case of interest, where $\tau_1 =0, \ldots, \tau_n = n-1$ and $d=2$. Significant connections between the matrices and their decompositions are additionally discussed in Section \ref{Kap. Regimeunabhaengige Untersuchung II}. We conclude the paper in Sections \ref{Kap: stoch. appl.} and 
 \ref{Kap: Resumee} with applications and open problems as mentioned before.

\section{Positive semidefinite symmetric matrices} \label{Kap: Positive semidefinite symmetric matrices}

Here we recapitulate facts related to a matrix $\Sigma_n$ of length power functionals (see Theorem \ref{Satz: Kov.matrix Laengen Potenz Funktional}), which are used all over the manuscript. Recall that $\Sigma_n$ has the following structure:
\[\left( \begin{array}{cccc}
 \lim \limits_{t \rightarrow \infty} \V(\tilde{L}_t^{(\tau_1)})&  \lim \limits_{t \rightarrow \infty} \C\text{ov}(\tilde{L}_t^{(\tau_1)}, \, \tilde{L}_t^{(\tau_2)}) &\ldots & \lim \limits_{t \rightarrow \infty} \C\text{ov}(\tilde{L}_t^{(\tau_1)}, \,\tilde{L}_t^{(\tau_n)})\\
 \lim \limits_{t \rightarrow \infty} \C\text{ov}(\tilde{L}_t^{(\tau_2)}, \,\tilde{L}_t^{(\tau_1)})&  \lim \limits_{t \rightarrow \infty} \V (\tilde{L}_t^{(\tau_2)}) &\ldots & \lim \limits_{t \rightarrow \infty} \C\text{ov}(\tilde{L}_t^{(\tau_2)}, \, \tilde{L}_t^{(\tau_n)})\\
\vdots& \vdots  &\hdots &\vdots \\
 \lim \limits_{t \rightarrow \infty} \C\text{ov}(\tilde{L}_t^{(\tau_n)}, \, \tilde{L}_t^{(\tau_1)}) &  \lim \limits_{t \rightarrow \infty} \C\text{ov}(\tilde{L}_t^{(\tau_n)}, \, \tilde{L}_t^{(\tau_2)})&\ldots &  \lim \limits_{t \rightarrow \infty} \V (\tilde{L}_t^{(\tau_n)}) \\ 
\end{array}\right)\in \R^{n \times n}.\] 

A sum $C=A+B$ of symmetric matrices $A, \, B$ is again symmetric and matrix properties of interest of $C$ can be obtained by the ones of $A,B$ as follows:

\begin{lem} 
\label{Lemma: Courant-Fisher}
Let $A, \, B \in \R^{n \times n}$ be two symmetric matrices and $C= A+B\in \R^{n \times n}$ with corresponding eigenvalues 
$\lambda_1^A \leq \ldots\leq \lambda_n^A$, $\lambda_1^B \leq \ldots \leq \lambda_n^B$ and $\lambda_1^C \leq \ldots \leq \lambda_n^C$.
Then:
\begin{enumerate}  
\item $\lambda_j^A + \lambda_1^B \leq \lambda_j^C \leq \lambda_j^A + \lambda_n^B \text{ for } j \in [n] $.
\item If $A$ is positive definite and $B$ is positive semidefinite, then $C$ is positive definite.
\item If $A$ and $B$ are positive semidefinite, then $\emph{rank}(C) \geq \emph{max}\{ \emph{rank}(A), \,\emph{rank}(B) \}.$
\end{enumerate}
\end{lem}
\begin{proof}
(i) follows from the theorem of Courant--Fischer. This yields then (ii) and (iii).
\end{proof}

Bounds for eigenvalues of positive definite
symmetric matrices can, e.g., be found in
\cite[(2.3)]{Wolkowicz} (upper bounds)
and \cite[(2.54)]{Wolkowicz} (lower bounds). More precisely, they are:
\begin{lem} 
\label{Lemma: allg. Schranken EW.e}
Let $A \in \R^{n \times n}$ be a positive definite symmetric matrix with eigenvalues $\lambda_1 \leq \ldots  \leq \lambda_n$, $m = \emph{tr}(A)/n$, and $s^2 = \emph{tr}(A^2)/n - m^2$. 
Then 
\[ 
\emph{det}(A)/\left(m+\frac{s}{\sqrt{n-1}}\right)^{n-1}
\leq \lambda_1 \leq \ldots \leq \lambda_n \leq m +s \sqrt{n -1}. \] \end{lem}

Matrix decompositions of $\Sigma_n$ are studied in Sections \ref{Kap: Superkrit. Reg.}--\ref{Kap: Krit. Reg.}. Recall that an \emph{LU decomposition} of a matrix $A \in \mathbb{R}^{n \times n}$ is an equation 
$PA= LU$,
where $L=(l_{ij})\in \mathbb{R}^{n \times n}$ is a normalized lower triangular matrix and $U=(u_{ij})\in \mathbb{R}^{n \times n}$ is an upper triangular matrix, which doesn't need to be normalized, i.e. also entries $u_{ii} = 0$ are allowed. Moreover, $P$ is a matrix obtained from the unit matrix by a certain number of row switches. For the next theorem see, e.g., \cite{RichterWick}.

\begin{thm} \label{Alg: LR}
Let $A \in \mathbb{R}^{n \times n} $ be a matrix. There exists an algorithm which terminates and returns the matrices $P$, $L$ and $U$ for a correct LU decomposition $ PA = L U$. If $A$ is regular and $P = I_n$, this decomposition is unique and $u_{ii} \neq 0$ for $ i \in [n]$.
\end{thm}

In this manuscript $I_n$ is the $n \times n$ unit matrix. Besides LU decompositions we also investigate a \emph{Cholesky decomposition} $A = GG^t$ of a  positive semidefinite symmetric matrix $A \in \mathbb{R}^{n \times n}$ into a product of a lower triangular matrix $G=(g_{ij})\in \mathbb{R}^{n \times n}$ and its transpose. See, e.g., \cite{RichterWick} or \cite[Theorem 4.2.6.]{GolubVanLoan} for a proof of:

\begin{Alg} 
\label{Algo: Korrektheit Cholesky Alg. pos. semidef.}
\begin{align*}
\textbf{Input:}& \text{  Matrix }A \in \mathbb{R}^{n \times n}
\text{ positive semidefinite and symmetric}. \\
\textbf{Output:}& \text{ Matrix } G= (g_{ij}) \in \mathbb{R}^{n \times n}  
\text{ s.t. }A = GG^t.\\
\textbf{Do:}&
\end{align*}
Compute $G=(g_{ij})\in \mathbb{R}^{n \times n }$  column by column as
\begin{align} \label{Eq: Eintrage Cholesky-Alg. pos. semidef.}
g_{ij} = \left\{
\begin{array}{ll}
0& \text{ for } \,i < j,\\
\sqrt{a_{ii}-  \sum_{k=1}^{i-1}g_{ik}^2} &\text{ for } \, i=j, \\
\frac{1}{g_{jj}}(a_{ij}- \sum_{k=1}^{j-1} g_{ik}g_{jk} )& \text{ for } \, i >j, \text{ and } g_{jj} \neq 0, \\
0&\text{ for } \,  i >j  \text{ and } g_{jj} =0.\\
\end{array}
\right.
\end{align}  If $A$ is positive definite, then $G$ is uniquely determined and $g_{ii}>0$ for all $i \in [n]$.  
\end{Alg}

An LU decomposition $A =LU$ and a Cholesky decomposition $A =G G^t$ 
of $A \in \mathbb{R}^{n \times n}$ are both decompositions into a lower and an upper triangular matrix. If $A$ is symmetric and positive definite, then the unique decompositions can be merged into another. Indeed:
\begin{enumerate}
    \item The matrix $G$ from the Cholesky-decompo\-sition of $A$ is given by 
$G = LD^{\frac{1}{2}}$, where $D^{\frac{1}{2}} = (\sqrt{d_{ij}}) \in \mathbb{R}^{n \times n}$ is a diagonal matrix with $d_{ii}= u_{ii}$.
    \item In the LU decomposition of $A$ one has $L = G\tilde{D}$  and $U = \tilde{D}^{-1}G^t$, where $\tilde{D} = (\tilde{d}_{ij}) \in \mathbb{R}^{n \times n}$ is a diagonal matrix with $\tilde{d}_{ii}= 1/g_{ii}$.
\end{enumerate}

\section{Supercritical Regime} \label{Kap: Superkrit. Reg.}

Recall Theorem \ref{Satz: Kov.matrix Laengen Potenz Funktional} and the notation used in (\ref{Eq: Kov.matrix Laengen Potenz Funktional}).
In this section $\Sigma_n$ is always considered in the supercritical regime (i.e. $t\delta_t^d \rightarrow \infty$) for $n\geq 2$. First examinations regarding eigenvalues and eigenspaces of $\Sigma_n$ are: 

\begin{exa} \label{Bsp: EW.e, ER.e superkrit. Reg.}  
We consider the asymptotic covariance matrix $\Sigma_2$ for any dimension $d$ with $\tau_2 > \tau_1 >- d/2$ and volume $V(W)$. 
A calculation of the characteristic polynomial yields
\begin{align*}
\chi(\Sigma_2) = (\lambda-\lambda_1)(\lambda-\lambda_2)  \in \R[\lambda]
\end{align*}
with eigenvalues $\lambda_1=0 \, \text{ and } \,\lambda_2 = \frac{V(W)d^2\kappa_d^2\left((\tau_1+d)^2+(\tau_2+d)^2\right)}{(\tau_1+d)^2(\tau_2+d)^2}.$
The eigenspaces are
\[
\text{eig}(\Sigma_2, \,\lambda_1) =  \{ r ( 
\frac{\tau_1+d}{\tau_2+d},
-1 )^t  \mid r \in \R \}
\text{ and }
\text{eig}(\Sigma_2, \,\lambda_2) = \{ r ( 
\frac{\tau_2+d}{\tau_1+d},
1 )^t \mid r \in \R\}. 
\]

\end{exa}
This examples can be generalized to arbitrary $n\geq 2$. We introduce the abbreviation:
\begin{align} \label{Eq: Abbreviation a_i}
   a_i =\tau_i+d \,\text{ for }\,  i \in [n]
   \text{ and }
   b=\sum_{k=1}^n \prod_{l \in [n] \setminus \{k\}}a_l^2.  
\end{align}

According to (\ref{Eq: Kov.matrix Laengen Potenz Funktional})
we have in the supercritical regime 
for the symmetric and positive semidefinite matrix $\Sigma_n$ that
\begin{equation}
\label{Eq: Kov.matrix Laengen Potenz Funktional supercritical}
\Sigma_n
=
V(W)d^2\kappa_d^2
\left( 
\begin{array}{cccc}
\frac{1}{a_1^2} & \frac{1}{a_1 a_2}  & \ldots & \frac{1}{a_1 a_n} \\
\frac{1}{a_1a_2} & \frac{1}{a_2^2} & \ldots  & \frac{1}{a_2 a_n}  \\
\vdots & \vdots& \ldots & \vdots \\
\frac{1}{a_1 a_n} & \frac{1}{a_2a_n} & \ldots & \frac{1}{a_n^2}\\
\end{array}
\right).
\end{equation}
\begin{thm} \label{Satz: EW.e, ER.e superkrit. Reg.} 
Let $\Sigma_n$ be as in (\ref{Eq: Kov.matrix Laengen Potenz Funktional}) in the supercritical regime for $n\geq2$. Then its eigenvalues are $\lambda_1 =0$, $\lambda_2 =V(W)d^2\kappa_d^2(\sum_{i=1}^n 1/ a_i^2 )$.
Let
\[
v_1 = ( \frac{a_1}{a_2},-1,0,\ldots,0)^t,   \ldots ,\, 
v_{n-1}=( \frac{a_1}{a_n},0,\ldots,0,-1,)^t,
\text{ and }
v_n = \left(a_n/a_1 ,\, a_n/a_2 ,\ldots, \, 1\right)^t.
\]
Then the eigenspaces of $\Sigma_n$ are
$\emph{eig}(\Sigma_n, \, \lambda_1)
= 
\langle 
v_1, \ldots, v_{n-1} 
\rangle \text{ and }
\emph{eig}(\Sigma_n, \,\lambda_2) =
\langle 
v_n
\rangle. 
$
\end{thm}

\begin{proof}
For $\lambda_2 $ one sees that 
$v_n$ is an eigenvector, since 
\begin{align*}
\left( \begin{array}{cccc}
\frac{1}{a_1^2} & \frac{1}{a_1 a_2}  & \ldots & \frac{1}{a_1 a_n} \\
\vdots & \vdots& \ldots & \vdots \\
\frac{}{a_1 a_n} & \frac{}{a_2a_n} & \ldots & \frac{}{a_n^2}\\
\end{array}\right) 
v_n 
&=
\left( \begin{array}{c}
\frac{a_n}{a_1^3} + \frac{a_n}{a_1a_2^2} + \cdots + \frac{1}{a_1a_n} \\
\vdots  \\
\frac{1}{a_1^2} +\frac{1}{a_2^2}+ \ldots + \frac{1}{a_n^2}  
\end{array}\right)
=
\left(\frac{1}{a_1^2}+ \frac{1}{a_2^2}+ \ldots + \frac{1}{a_n ^2}\right) 
v_n.
\end{align*}  
A similar computation shows that the $(n-1)$-many (obviously) linearly independent vectors 
$v_1,   \ldots, v_{n-1}$ are eigenvectors for $\lambda_1$. This concludes the proof.
\end{proof}

\begin{cor} 
\label{Kor: Rang superkrit. Reg.}
With the the assumptions of Theorem \ref{Satz: EW.e, ER.e superkrit. Reg.} 
it follows:~(i) $\emph{rank}(\Sigma_n) = 1$; (ii) $\emph{det}(\Sigma_n) =0$;
(iii) $\Sigma_n$ is positive semidefinite but, not positive definite;
(iv) The characteristic polynomial of $\Sigma_n$ is $\chi(\Sigma_n)= (\lambda-\lambda_1)^{n-1}(\lambda-\lambda_2) \in \R[\lambda] .$
\end{cor}
The preceding outcome of the rank and definiteness was already found in \cite[Proposition 3.4]{Reitzner}.
The results related to eigenvalues and eigenspaces of $\Sigma_n$ in Theorem \ref{Satz: EW.e, ER.e superkrit. Reg.} lead to the next fact. 

\begin{cor} 
\label{Kor: Diag.matrix, char. Poly. superkrit. Reg.}
With the the assumptions of Theorem \ref{Satz: EW.e, ER.e superkrit. Reg.} set  $D=  \diag(\lambda_1,\dots,\lambda_1,\lambda_2)\in \R^{n \times n}$
and $S \in \R^{n \times n}$ with column vectors $v_1, \dots,v_n$.
Then $\Sigma_n$ is similar to $D$ and $D=S^{-1} \Sigma_n S$, where $S^{-1} = (\tilde{s}_{ij})\in \R^{n \times n}$ with entries
\begin{align*} 
\tilde{s}_{ij} = \left\{
\begin{array}{ll}
\frac{\prod_{k=1}^n a_k^2}{a_{j}a_n b}&\text{ for }\,i=n,\\
-\frac{\sum_{k \in [n] \setminus \{i\}} a_{i}^2\prod_{l\in [n] \setminus \{k, \, i\}} a_{l}^2}{b}&\text{ for}\,j=i+1,\\
\frac{\prod_{k=1}^n a_k^2}{a_{i+1}a_j b}& \text{ else.}\\
\end{array}
\right.
\end{align*} 
\end{cor} 

\begin{proof}
Let $\tilde{S}$ be the matrix with entries $\tilde{s}_{ij}$. There are four distinct cases for entries of $S\tilde{S}$. The first one arises by the multiplication of the first row of $S$ with the first column of $\tilde{S}$:
\begin{align*}
\left(\sum_{k=1}^{n-1} \frac{a_1}{a_{k+1}} \cdot\tilde{s}_{k1}\right) + \frac{a_n}{a_1}\cdot \tilde{s}_{n1} 
&=
\left(
\sum_{k=1}^{n-1} 
\frac{a_1\prod_{m=1}^n a_m^2}
{a_{k+1}^2a_1 b} 
\right)
+ 
\frac{a_n\prod_{m=1}^n a_m^2}
{a_{1}^2a_n b}=1.
\end{align*}
The multiplication of the first row of $S$ with the $j$-th column of $\tilde{S}$ for $j \neq 1$ is
\begin{align*}
&\sum_{k\in [n-1]\setminus \{j-1\}} \frac{a_1}{a_{k+1}}\cdot \tilde{s}_{kj} + \frac{a_1}{a_{j}} \cdot\tilde{s}_{j-1j}+ \frac{a_n}{a_1}\cdot \tilde{s}_{nj} 
\\
&= 
\sum_{k\in [n-1]\setminus \{j-1\}}  
\frac{a_1a_j\prod_{m \in [n] \setminus \{j, \, k+1\}} a_m^2}
{b} 
- 
\frac{a_1a_j\sum_{k \in [n] \setminus \{j\}} \prod_{l\in [n] \setminus \{k, \, j\}} a_{l}^2}
{b} 
+ 
\frac{a_1a_j\prod_{k\in [n]\setminus \{1, j\}} a_k^2}
{b}
\end{align*}

and this is easily seen to be 0.
The multiplication of the $i$-th row of $S$ with the $j$-th column of $\tilde{S}$ for $i \neq 1$ and $j\neq i$ gives 

\begin{align*}
    &- \tilde{s}_{i-1j} + \frac{a_n}{a_i}\cdot \tilde{s}_{nj} =-\frac{\prod_{k=1}^n a_k^2}{a_{i}a_j b}+ \frac{a_n}{a_i}\cdot \frac{\prod_{k=1}^n a_k^2}{a_{j}a_n b} =0.
\end{align*}
Finally, the multiplication of the $i$-th row of $S$ with the $j$-th column of $\tilde{S}$ for $i=j \neq 1$ yields
\begin{align*}
    - \tilde{s}_{i-1j} + \frac{a_n}{a_i} \cdot\tilde{s}_{nj} &=
\frac{\sum_{k \in [n] \setminus \{i\}}a_{i}^2 \prod_{l\in [n] \setminus \{k, \, i\}} a_{l}^2}
{b}
+ 
\frac{a_n}{a_i} 
\cdot
\frac{\prod_{k=1}^n a_k^2}
{a_{j}a_n b} =1 .
\end{align*}
All in all thereby arises the unit matrix. Hence, $ \tilde{S} = S^{-1}$. 
\end{proof}
  
In the preceded result we found the \textit{Schur decomposition} 
\[\Sigma_n = SDS^{-1}
\]
of $\Sigma_n$. See \cite{Horn} for more details. Other interesting decompositions are:

\begin{thm}  
\label{Satz: LR-Zerleg. superkrit. Reg.}
Let $\Sigma_n$ be as in (\ref{Eq: Kov.matrix Laengen Potenz Funktional}) in the supercritical regime for $n\geq2$. 
\begin{enumerate}
\item 
An LU decomposition $\Sigma_n =LU$ of $\Sigma_n$ is given by  $L= (l_{ij})\in \R^{n \times n}$ and $U=(u_{ij})\in \R^{n \times n}$ with entries
\begin{align*} 
l_{ij} = \left\{
\begin{array}{ll}
1&\text{ for }\,i=j,\\
\frac{a_1}{a_i} & \text{ for }\,i >1  \text{ and }  j=1,\\
0& \text{ else},\\
\end{array}
\right.\, \text{ and }\,
u_{ij} = \left\{
\begin{array}{ll}
\frac{V(W)(d\kappa_d)^2}{a_1a_j} & \text{ for }\,i=1,\\
0& \text{ else}.\\
\end{array}
\right.
\end{align*}
\item
A Cholesky decomposition $ \Sigma_n = G G^t$ of $\Sigma_n$ is given by  $G= (g_{ij})\in \R^{n \times n}$ with entries 
\begin{align*} 
g_{ij} = \left\{
\begin{array}{ll}
\frac{\sqrt{V(W)}d\kappa_d}{a_i}& \text{ for } j=1,\\
0& \text{ else}.\\
\end{array} \right.
\end{align*}
\item
A matrix root $B=(b_{ij})\in \R^{n \times n}$ of $\Sigma_n$ 
(i.e., $B^2=\Sigma_n$) has entries
\begin{align*} 
b_{ij} = \frac{\sqrt{V(W)}d\kappa_d\prod_{k=1}^na_k}{a_ia_j \sqrt{\sum_{i=1}^n\prod_{k \in [n] \setminus \{i\}}a_k^2}}.
\end{align*}
\end{enumerate}
\end{thm}
\begin{proof}
(i) 
Computing the entries $\sigma_{ij}$ of the given product $LU$ leads to
\[ \sigma_{1j} = 1 \cdot \frac{V(W) (d \kappa_d) ^2}{a_1 a_j}  = \frac{V(W) (d \kappa_d) ^2}{a_1 a_j} \text{ for $j \in [n]$ } \text{ and}\]
\[ \sigma_{ij} = \frac{a_1}{a_i} \cdot \frac{V(W) (d \kappa_d) ^2}{a_1 a_j}  = \frac{V(W) (d \kappa_d) ^2}{a_i a_j} \text{ for $i > 2$ and $j \in [n]$, }\]
which are exactly the entries of $\Sigma_n$. 

(ii)
We determine the entries of $G$ column by column by using formula (\ref{Eq: Eintrage Cholesky-Alg. pos. semidef.}). 
Without a computation one obtains $g_{ij} =0$ for $j>i$. 
Additionally, we get the first column of $G$ by 
\[ 
g_{11} = \sqrt{\sigma_{11}} = \sqrt{V(W)}d\kappa_d/a_1 \, \text{ and } \, g_{i1} = \sigma_{i1} /(\sqrt{V(W)}d\kappa_d/a_1)  = \frac{\sqrt{V(W)}d\kappa_d}{a_i} \,\text{ for } \,i \geq 2.
\]
For $i \geq 2$, one has $g_{ii} =\sqrt{(d\kappa_d)^2 V(W)/a_i^2 - g_{i1}^2} = 0$, so the remaining columns  are $0$.

(iii)
Let $S,\, D$ and $S^{-1}$  be the matrices from Corollary \ref{Kor: Diag.matrix, char. Poly. superkrit. Reg.}. Let $B= SD^{1/2}S^{-1}$. Then $ BB=SDS^{-1}=\Sigma_n.$ The product $SD^{1/2}=R=(r_{ij})\in \R^{n \times n}$ has entries 
\[
r_{ij} = \left\{
\begin{array}{ll}
\frac{\sqrt{V(W)}d\kappa_d a_n\sqrt{\sum_{i=1}^n\prod_{k \in [n] \setminus \{i\}}a_k^2}}{a_i\prod_{k\in [n]}a_k}& \text{ for } j=n,\\
0& \text{ else}.\\
\end{array} \right.
\]
The claimed entries for $B$ are determined from the product $RS^{-1}$.
\end{proof}

The next example illustrates the decompositions found above for a vector of length $n=3$:

\begin{exa} \label{Bsp: Zerlegungen superkrit. Reg.}
Consider $\Sigma_3$ for $d =2,\, V(W)=1, \, \tau_1 =0,\, \tau_2 = 1$ and $\tau_3 =2$:
\[\Sigma_3 =4 \pi^2 \left( \begin{array}{ccc}
1/4& 1/6  & 1/8  \\
1/6 & 1/9 & 1/12\\
1/8 & 1/12& 1/16 \\
\end{array}\right) .\]
 According to Theorem \ref{Satz: LR-Zerleg. superkrit. Reg.}(i) an LU decomposition of $\Sigma_3$ is given by
\[  LU =4 \pi^2  \left( \begin{array}{ccc}
1& 0  & 0 \\
4/6 & 1 & 0 \\
4/8 & 0& 1\\
\end{array}\right) \left( \begin{array}{ccc}
1/4& 1/6  & 1/8  \\
0 & 0 & 0 \\
0 & 0& 0 \\
\end{array}\right).\] 
Theorem \ref{Satz: LR-Zerleg. superkrit. Reg.}(ii) leads to a Cholesky decomposition of $\Sigma_3$:
\[  GG^t= 4 \pi^2  \left( \begin{array}{ccc}
1/2& 0  & 0\\
1/3 & 0 & 0\\
1/4 & 0& 0\\
\end{array}\right) \left( \begin{array}{ccc}
1/2& 1/3  & 1/4\\
0 & 0 & 0 \\
0 & 0& 0\\
\end{array}\right).
\]
Theorem \ref{Satz: LR-Zerleg. superkrit. Reg.}(iii) yields $\Sigma_3$ as the following product of matrix roots: 

\[
BB= 4 \pi^2  \left( \begin{array}{ccccc}
6/\sqrt{244}& 4/\sqrt{244}  & 3/\sqrt{244} \\
4/\sqrt{244}& 8/3\sqrt{244} & 2/\sqrt{244}\\
3/\sqrt{244} & 2/\sqrt{244}& 3/2\sqrt{244}  \\
\end{array}\right) \left( \begin{array}{ccccc}
6/\sqrt{244}& 4/\sqrt{244}  & 3/\sqrt{244} \\
4/\sqrt{244}& 8/3\sqrt{244} & 2/\sqrt{244}\\
3/\sqrt{244} & 2/\sqrt{244}& 3/2\sqrt{244}  \\
\end{array}\right).
\] 
\end{exa}

\section{Subcritical Regime} \label{Kap: Subkrit. Reg.}

In this section $\Sigma_n$ is always considered in the subcritical regime (i.e. $t\delta_t^d \rightarrow 0$) for $n\geq 2$. Recall that then this matrix has the following structure:
\begin{equation}
\label{Eq: Kov.matrix Laengen Potenz Funktional subcritical}
\Sigma_n
=
V(W) \left(\frac{d\cdot \kappa_d}{2( \tau_i + \tau_j +d) }\right) \in  \R^{n \times n}.
\end{equation}
By re-scaling and substituting 
\[
x_i:=\tau_i + d/2  \,\text{ for } \,i \in [n]
\]
one obtains the following \emph{Cauchy matrix} (see, e.g., \cite{Horn}): 
\begin{align}
\label{Eq: Kov.matrix Omega subkrit. Reg}
\Omega_n 
= 
(\omega_{ij}):=
\frac{2}{d\cdot \kappa_dV(W)} \Sigma_n
=
\biggl( \frac{1}{\tau_i + \tau_j + d }  \biggr)
=
\biggl( \frac{1}{x_i- (-x_j)}  \biggr)\in  \R^{n \times n}.
\end{align} 

It follows from \cite[0.9.12.1]{Horn} that:
\begin{lem} 
\label{Satz: Det. mod. Matrix subkrit. Reg.}
Let $\Omega_n$ be as in (\ref{Eq: Kov.matrix Omega subkrit. Reg}). Then
\[ 
\emph{det} (\Omega_n) = 
\frac{\prod_{1 \leq i < j \leq n} (x_i -x_j)^2}{ \prod_{ 1 \leq i, \,j \leq n }( x_i +x_j ) } >0.
\]
\end{lem}

\begin{cor} \label{Kor: Det. Kov.matrix subkrit. Reg.} 
Let $\Sigma_n $ be as in (\ref{Eq: Kov.matrix Laengen Potenz Funktional}) in the subcritical regime for $n\geq2$. Then: 
\begin{align} \label{Eq: Det. Sigma subkrit. Reg}
\emph{det} (\Sigma_n) = \biggl(\frac{V(W)d\kappa_d}{2}\biggr)^n
\frac{\prod_{1 \leq i < j \leq n} (x_i -x_j)^2}{ \prod_{ 1 \leq i, \,j \leq n }( x_i +x_j ) } >0.
\end{align}
In particular,  
(i) $\emph{rank}(\Sigma_n) = n$;
(ii) $\Sigma_n$ is positive definite.
\end{cor}
 The aforementioned conclusion about the rank and the definiteness was given before in \cite[Proposition 3.4]{Reitzner}. The inverse for an arbitrary Cauchy matrix is explicitly known, e.g., by \cite[Corollary 3.1]{Finck} since there are special cases of Cauchy-Vandermonde matrices. For $\Sigma_n$ this implies:

\begin{prop} 
\label{Satz: Inv. subkrit. Reg.}
Let $\Sigma_n$ be as in (\ref{Eq: Kov.matrix Laengen Potenz Funktional}) in the subcritical regime for $n \geq 2$. The inverse $\hat{\Sigma}_{n} = (\hat{\sigma}_{ij})\in \mathbb{R}^{n \times n}$ of $\Sigma_n$ has entries
\[ \hat{\sigma}_{ij}= 
\left\{
\begin{array}{ll}
-\frac{ 8x_ix_j(x_i+x_j)\prod_{ l \in [n] \setminus \{i, \, j\},\, k \in \{i, \, j\}} (x_k+x_l) }{V(W)d\kappa_d(x_j-x_i)^2\prod_{l \in  [n]\setminus \{i, \, j\}  ,\, k \in \{i, \, j\}} (x_k-x_l)} & \text{ for } i \neq j, \\
 \frac{4x_i \prod_{k \in  [n] \setminus \{i\}} (x_k+x_i)^2}{V(W)d\kappa_d\prod_{k \in  [n] \setminus \{i\}} (x_k-x_i)^2} & \, \textrm{ for } i =j. \\
\end{array}
\right.\]
\end{prop}
   
Next we consider eigenvalues and eigenspaces of $\Sigma_n$. For this it is useful to observe that for an 
eigenvalue $\lambda\in \R$ and a corresponding 
eigenvector $v \in \R^n$ of $\Omega_n$ one has

\begin{equation}
\label{Bem: EWe, EVen Omega, Sigma subkrit. Reg.}
\Sigma_n v =
(V(W)d\kappa_d/2)\Omega_n v = 
(V(W)d\kappa_d/2) \lambda v.
\end{equation}

\begin{exa} \label{Bsp: char. Poly. Sigma2 subkrit. Reg.}
Consider $\Omega_2$ as in (\ref{Eq: Kov.matrix Omega subkrit. Reg}). Its characteristic polynomial $\chi(\Omega_2)\in \R[\lambda]$ is
\begin{align*}
\chi(\Omega_2)&= \text{det}  \left( \begin{array}{ccccc}
\frac{1}{x_1+ x_1}- \lambda& \frac{1}{x_1 + x_2}  \\
\frac{1}{x_1+x_2}& \frac{1}{x_2+x_2} - \lambda \\
\end{array}\right) = \lambda^2 - \lambda\biggl(\frac{1}{2x_1} + \frac{1}{2x_2}\biggr)  + \frac{(x_1-x_2)^2}{2x_12x_2(x_1+x_2)^2} . 
\end{align*} 
Thus, according to (\ref{Bem: EWe, EVen Omega, Sigma subkrit. Reg.}) all eigenvalues of $\Sigma_2$ are
\begin{align*}
\lambda_{1,\,2} = \frac{V(W)d\kappa_d}{8x_1x_2(x_1+x_2)}\biggl((x_1+x_2)^2 \mp \sqrt{x_1^4+14x_1^2x_2^2+x_2^4}  \biggr). 
\end{align*}
\end{exa}

Experiments and further examples show that it is difficult to determine explicitly the eigenspaces of $\Sigma_n$ for  $n \geq 3$. Instead, we prove in the following a formula for the characteristic polynomial and afterwards bounds for the eigenvalues of $\Sigma_n$. Similar as above for eigenvalues and eigenvectors,
we observe the following relationship between  $\chi(\Omega_n)$ and $\chi(\Sigma_n)$. Let $\Omega_n$ be as in (\ref{Eq: Kov.matrix Omega subkrit. Reg}) with given characteristic polynomial
\[
\chi(\Omega_n) = a_n \lambda^n +  \ldots + a_1\lambda^1 +a_0\in \mathbb{R}[\lambda].
\]
Then $\chi(\Sigma_n)=\text{det}\left( \Sigma_n - \lambda I_n \right)=\left(\frac{V(W)d\kappa_d}{2} \right)^n \text{det}( \Omega_n -\frac{2}{V(W)d\kappa_d} \lambda I_n) $ equals to
\begin{equation}
\label{Bem: char. Poly. Omega, Sigma subkrit. Reg.}
a_n \lambda^n +  a_{n-1}\lambda^{n-1}\left(\frac{V(W)d\kappa_d}{2}\right) +\ldots +a_1 \lambda\left(\frac{V(W)d\kappa_d}{2}\right)^{n-1}+a_0 \left(\frac{V(W)d\kappa_d}{2}\right)^n.
\end{equation}

The following lemma is a variation of a well-known fact for the characteristic polynomial of an arbitrary quadratic matrix, which is useful for the investigated matrices in this paper. For example, it is an immediate consequence of \cite[Theorem 1.2.16.]{Horn}. 

\begin{lem}
\label{Satz: char. Poly. OmegaRS subkrit. Reg.}
Let
$\Omega_{R,\,S}=\left(z_{r_is_j}\right)\in \mathbb{R}^{n \times n}$ 
with
$R = (r_1,\ldots, r_n), \, S = (s_1,\ldots, s_n) \in \mathbb{N}^n$
and characteristic polynomial $\chi(\Omega_{R,\,S})  = (-1)^n\lambda^{n}+ a_{R,\,S}^{n-1,\,n}\lambda^{n-1}+  \ldots+a_{R,\,S}^{1,\,n} \lambda^1 + a_{R,\,S}^{0,\,n}  \in \mathbb{R}[\lambda].$ Then  the coefficients  $a_{R,\,S}^{k,\,n}$ for $k \in [n-1]$ can be computed via the formula
\[ 
a_{R,\,S}^{k,\,n} =
(-1)^k 
\sum_{1\leq i_1< \ldots <i_{n-k}\leq n} 
a_{(r_{i_1}, \ldots, r_{i_{n-k}}),\,(s_{i_1}, \ldots ,s_{i_{n-k}})}^{0,\,n-k}.
\] 
\end{lem}

Note that  
$a_{(r_{i_1} ,\ldots, r_{i_{n-k}}),\,(s_{i_1}, \ldots, s_{i_{n-k}}) }^{0,\,n-k}$ is  the determinant of the matrix $(z_{r_{i_u}s_{i_v}})_{u , \, v \in [n-k]}$. Using the previous lemma and the formula for the determinant in Lemma \ref{Satz: Det. mod. Matrix subkrit. Reg.} we can calculate $\chi (\Omega_n) = a_n^{(n)} \lambda ^n + \ldots + a_n^{(0)}$ of $\Omega_n$ given as in (\ref{Eq: Kov.matrix Omega subkrit. Reg}), since  $\Omega_n=\Omega_{[n],[n]} $ with $z_{r_is_j} = \frac{1}{x_{r_i}+x_{s_j}}$ and $a_n^{(k)}= a_{(1, \ldots, n),\,(1, \ldots, n)}^{k,\,n}$. Taking  additionally (\ref{Bem: char. Poly. Omega, Sigma subkrit. Reg.}) into account, we obtain:
 
\begin{cor} 
\label{Satz: char. Poly. Sigma subkrit. Reg. }
Let $\Sigma_n$ be as in (\ref{Eq: Kov.matrix Laengen Potenz Funktional}) in the subcritical regime for $n \geq 2$ with $\chi(\Sigma_n) = b_n^{(n)} \lambda^n+b^{(n-1)}_n \lambda^{n-1}  + \ldots + b_n^{(1)} \lambda^1 + b_n^{(0)} \in \mathbb{R}[\lambda].$ The coefficients of $\chi(\Sigma_n)$ are
\begin{align*}
b_n^{(k)}&= (-1)^k\left(\frac{V(W)d\kappa_d}{2} \right)^{n-k}  \left(\sum_{1 \le i_1< \ldots <i_{n-k} \leq n} \frac{\prod_{l ,\,k\in \{i_1, \ldots, i_{n-k} \} ,\, l<k}(x_l-x_k)^2}{\prod_{l,\, k \in \{i_1, \ldots, i_{n-k} \}} (x_l+x_k)}  \right), k \in [n-1]. 
\end{align*}
\end{cor}

Determining the eigenvalues of $\Sigma_n$ explicitly is an open problem.
However, we have: 

\begin{thm} \label{Satz: Schranken EWe Sigma subkrit. Reg.}
Let $\Sigma_n$ be as in (\ref{Eq: Kov.matrix Laengen Potenz Funktional}) in the subcritical regime for $n \geq 2$ with eigenvalues $ \lambda_1 \leq \lambda_{2} \leq \ldots \leq \lambda_n$. We have $\text{\b{S}}_n \leq \lambda_1 \leq \ldots \leq \lambda_n \leq \bar{S}_n$ with 
\begin{align*}
\bar{S}_n&=\frac{V(W)d\kappa_d\left(\sum_{i=1}^n \frac{1}{2x_i}+\sqrt{(n-1)\left(\sum_{l=1}^n \sum_{k=1}^n \frac{n}{(x_k+x_l)^2}-\left(\sum_{i=1}^n \frac{1}{2x_i}\right)^2\right)}\right)}{2n}  \text{ and }\\
\text{\b{S}}_n &= \frac{V(W)d\kappa_d\prod_{1 \leq i < j \leq n} (x_i -x_j)^2}{ 2\prod_{ 1 \leq i, \,j \leq n }( x_i+x_j ) \left(\frac{\sum_{i=1}^n \frac{\sqrt{n-1}}{2x_i}+\sqrt{\sum_{l=1}^n \sum_{k=1}^n \frac{n}{(x_k+x_l)^2}-\left(\sum_{i=1}^n \frac{1}{2x_i}\right)^2}}{n\sqrt{n-1}} \right)^{n-1} }. 
\end{align*}
\end{thm}
\begin{proof} 
Apply Lemma \ref{Lemma: allg. Schranken EW.e} to the matrix $\Omega_n$ according to (\ref{Eq: Kov.matrix Omega subkrit. Reg}). We have $ m =\sum_{i=1}^n (2nx_i)^{-1}.$ The matrix $\Omega_n^2$ has entries $\sum_{k=1}^n ((x_k+x_i)(x_k+x_j))^{-1}$ at places $(i,\,j)$ for $i,\, j \in [n]$. Hence, 
$\text{tr}( \Omega_n^2) = \sum_{l=1}^n \sum_{k=1}^n (x_k+x_l)^{-2}$ and, thus,
\[ s^2 = \frac{\sum_{l=1}^n \sum_{k=1}^n \frac{1}{(x_k+x_l)^2}}{n}- \left(\frac{\sum_{i=1}^n \frac{1}{2x_i}}{n}\right)^2 = \frac{\sum_{l=1}^n \sum_{k=1}^n \frac{n}{(x_k+x_l)^2} - \left(\sum_{i=1}^n \frac{1}{2x_i}\right)^2}{n^2}. \] 

Using the results for $\text{det}(\Omega_n)$ (according to Lemma \ref{Satz: Det. mod. Matrix subkrit. Reg.}), $m$ and $s$ in Lemma \ref{Lemma: allg. Schranken EW.e} and applying (\ref{Bem: EWe, EVen Omega, Sigma subkrit. Reg.}) by adding the multiplicative constant $V(W)d\kappa_d/2$ leads to the claimed bounds.
\end{proof} 
   
The given bounds in  Theorem \ref{Satz: Schranken EWe Sigma subkrit. Reg.} 
can be tight and appear as the smallest and largest eigenvalue of $\Sigma_n$. For $n=2$ this occurs for any $d, \,V(W)$ and $\tau_1<\tau_2$, which follows from the description of $\Omega_2$ in (\ref{Eq: Kov.matrix Omega subkrit. Reg}) and the known eigenvalues of $\Sigma_2$ from  Example \ref{Bsp: char. Poly. Sigma2 subkrit. Reg.}. We have
\begin{align*}
\bar{S}_2  
= \frac{V(W)d\kappa_d}{2} \left( \frac{x_2+x_1}{4x_1x_2} + \sqrt{\frac{14x_1^2x_2^2+x_2^4+x_1^4}{16x_1^2x_2^2(x_1+x_2)^2} }\right)
 = \lambda_2
\end{align*}
and for $\emph{\b{S}}_2$ one gets

\begin{align*}
&\frac{V(W)d\kappa_d}{2}\cdot\frac{\frac{(x_1-x_2)^2}{4x_1x_2(x_1+x_2)^2}}{\frac{(x_1+x_2)^2}{4x_1x_2(x_1+x_2)}+ \sqrt{\frac{14x_1^2x_2^2+x_1^4+x_2^4}{16x_1^2x_2^2(x_1+x_2)^2}}} =\frac{V(W)d\kappa_d}{2}\cdot \frac{\frac{(x_1-x_2)^2}{4x_1x_2(x_1+x_2)^2}}{\frac{(x_1+x_2)^2+ \sqrt{14x_1^2x_2^2+x_1^4+x_2^4}}{\sqrt{16x_1^2x_2^2(x_1+x_2)^2}} } \\
&=\frac{V(W)d\kappa_d}{2}\cdot \frac{\left((x_1+x_2)^2+\sqrt{14x_1^2x_2^2+x_1^4+x_2^4}\right)\left((x_1+x_2)^2-\sqrt{14x_1^2x_2^2+x_1^4+x_2^4}\right)}{4x_1x_2(x_1+x_2)\left( (x_1+x_2)^2+\sqrt{14x_1^2x_2^2+x_1^2+x_2^2}\right)}  = \lambda_1.
\end{align*}

It would be quite interesting to investigate in a further analysis, if $\emph{\b{S}}_n = \lambda_1$ and $\bar{S}_n= \lambda_n$ is true for $n\geq 3$ or whether the bounds are not tight in general. 

\begin{conj} \label{Verm: verschied. EWe subkrit. Reg.}
Let $\Sigma_n$ be as in (\ref{Eq: Kov.matrix Laengen Potenz Funktional}) in the subcritical regime for $n \geq 2$ with eigenvalues $\lambda_1 \leq \ldots \leq \lambda_n$. Then $\lambda_i \neq \lambda_j $ for all $i \neq j \in [n]$.
\end{conj}

Due to the structure of $\Sigma_n$ we determine its LU decomposition as: 

\begin{thm} \label{Satz: LR-Zerleg. subkrit. Reg.}
Let $\Sigma_n = (\sigma_{ij})$ be as in (\ref{Eq: Kov.matrix Laengen Potenz Funktional}) in the subcritical regime for $n \geq 2$. The $LU$ decomposition $\Sigma_n= L U$ is given by 
$L = (l_{ij})$ and $U = (u_{ij})$  with entries 
\begin{align} \label{Eq: L-Matrix subkrit. Reg.}
l_{ij} = \left\{
\begin{array}{ll} 
0 &  \text{ for }\, i < j \\
\frac{2x_j \prod_{k=1}^{j-1} (x_j + x_k)(x_i - x_k)}{\prod_{ k = 1}^{j-1} ( x_j - x_k ) \prod_{k=1}^{j}(x_i+ x_k)}& \text{ for }\, i\geq j, \\
\end{array}
\right. 
,
u_{ij} = \left\{
\begin{array}{ll}
0 &  \text{ for } \,i > j \\
\frac{V(W)d\kappa_d\prod_{k=1}^{i-1}(x_i-x_k)(x_j-x_k)}{2\prod_{ k = 1}^{i-1} ( x_i + x_k) \prod_{k=1}^{i}(x_j + x_k)}& \text{ for } \,i\leq j. \\
\end{array}
\right.  
\end{align} 
\end{thm}

\begin{proof} 

We first show the following formula for any $q \in \N, q\geq 1$ and $x_i, x_j, x_1, \ldots, x_q \in \R$: 
\begin{align} \label{Eq: Formal Produkt}
\nonumber\prod_{m=1}^{q} (x_i - x_m)(x_j-x_m) &=  \prod_{m=1}^q(x_i+x_m)(x_j+x_m)  \\
&-\sum_{m=1}^{q} 2x_m(x_i+x_j)\prod_{k =m+1}^q(x_i+x_k)(x_j+x_k) \prod_{k=1}^{m-1}(x_i-x_k)(x_j-x_k).
\end{align} 
We perform an induction on $q\geq 1$, where the base case $q=1$ 
is trivial.
Next one sees
\begin{align*}
&
\prod_{m=1}^q (x_i - x_m)(x_j-x_m) = (x_i -x_q)(x_j-x_q) \prod_{m=1}^{q-1} (x_i -x_m)(x_j-x_m) 
\\
=& 
(x_i +x_q)(x_j+x_q)\prod_{m=1}^{q-1} (x_i -x_m)(x_j-x_m)  - (2x_ix_q+ 2x_jx_q)\prod_{m=1}^{q-1} (x_i -x_m)(x_j-x_m) 
\\
\overset{\textrm{i.h.}}{=}
& 
- 2x_q(x_i+x_j)\prod_{m=1}^{q-1} (x_i-x_m)(x_j-x_m) + (x_i + x_q)(x_j+x_q)\Biggl(\prod_{m=1}^{q-1} (x_i+x_m)(x_j+x_m)  
 \\
 &- \sum_{m=1}^{q-1} 2x_m(x_i+x_j)\prod_{k = m+1}^{q-1} (x_i+x_k)(x_j+x_k) \prod_{k = 1}^{m-1}(x_i-x_k)(x_j-x_k) \Biggr)
 \\
=& -  2x_q(x_i+x_j) \prod_{m=1}^{q-1} (x_i-x_m)(x_j-x_m)
  +  \prod_{m=1}^{q} (x_i+x_m)(x_j+x_m)   \\
  &- \sum_{m=1}^{q-1} 2x_m(x_i+x_j) \prod_{k = m+1}^{q} (x_i+x_k)(x_j+x_k) \prod_{k = 1}^{m-1}(x_i-x_k)(x_j-x_k)
 \\
   =& \prod_{m=1}^{q} (x_i+x_m)(x_j+x_m)  - \sum_{m=1}^{q} 2x_m(x_i+x_j) \prod_{k = m+1}^{q} (x_i+x_k)(x_j+x_k) \prod_{k = 1}^{m-1}(x_i-x_k)(x_j-x_k).
\end{align*}
We compute the entries $\sigma_{ij} = \sum_{m=1} ^n l_{im}u_{mj}$ of the product $LU$. Let $q = \min \{i,j\}$. It is $ l_{im}u_{mj} = 0$ for $m \in \{ q+1, \ldots , n\}$. Thus, 
\begin{align*}
    \sigma_{ij} &= \sum_{m=1} ^q l_{im}u_{mj} =  \sum_{m=1} ^q  \frac{2x_m \prod_{k=1}^{m-1} (x_m + x_k)(x_i - x_k)}{\prod_{ k = 1}^{m-1} ( x_m - x_k ) \prod_{k=1}^{m}(x_i+ x_k)}\cdot \frac{V(W)d\kappa_d\prod_{k=1}^{m-1}(x_m-x_k)(x_j-x_k)}{2\prod_{ k = 1}^{m-1} ( x_m + x_k) \prod_{k=1}^{m}(x_j + x_k)} \\
    &= \frac{V(W) d\kappa_d}{2} \sum_{m=1} ^q  \frac{2x_m \prod_{k=1}^{m-1}(x_i - x_k)}{\prod_{k=1}^{m}(x_i+ x_k)} \cdot \frac{\prod_{k=1}^{m-1}(x_j-x_k)}{ \prod_{k=1}^{m}(x_j + x_k)} \\
    &= \frac{V(W) d\kappa_d}{2} \cdot \frac{\sum_{m=1} ^q 2x_m \prod_{k=1}^{m-1}(x_i - x_k)(x_j-x_k)\prod_{k=m+1}^{q}(x_i+ x_k) (x_j + x_k)}{\prod_{k=1}^{q}(x_i+ x_k) (x_j + x_k)} \\
    & \overset{(\ref{Eq: Formal Produkt})}{=}  \frac{V(W) d\kappa_d}{2}  \cdot \frac{\prod_{m=1} ^q (x_i+x_m)(x_j+x_m)-\prod_{m=1} ^q(x_i-x_m)(x_j-x_m)}{(x_i+x_j)\prod_{k=1}^{q}(x_i+ x_k) (x_j + x_k)} \\
    & =\frac{V(W) d\kappa_d}{2} \cdot \frac{\prod_{m=1} ^q (x_i+x_m)(x_j+x_m)}{(x_i+x_j)\prod_{k=1}^{q}(x_i+ x_k) (x_j + x_k)}
    =\frac{V(W) d\kappa_d}{2(x_i+x_j)}.
\end{align*}

These are the entries of $\Sigma_n$ given as in (\ref{Eq: Kov.matrix Laengen Potenz Funktional subcritical}).
\end{proof}  

\begin{cor} \label{Satz: Cholesky-Zerleg. subkrit. Reg.}
Let $\Sigma_n = (\sigma_{ij})$ be as in (\ref{Eq: Kov.matrix Laengen Potenz Funktional}) in the subcritical regime for $n \geq 2$. Then the Cholesky decomposition $ \Sigma_n =  GG^t $ is given by $G = (g_{ij})$ with entries
\begin{align*} 
g_{ij} = \left\{
\begin{array}{ll}
0 & \, \emph{ for } i < j, \\
\sqrt{\frac{V(W)d\kappa_d}{2}}\frac{\sqrt{2x_j}\prod_{k=1}^{j-1}(x_i-x_k)}{\prod_{k=1}^j (x_i+x_k)}  & \, \emph{ for } i\geq j. \\
\end{array}
\right. 
\end{align*}
\end{cor}
 
\begin{proof}
There exists a unique Cholesky decomposition of $\Sigma_n$ due to Algorithm \ref{Algo: Korrektheit Cholesky Alg. pos. semidef.}. According to the remark after this algorithm, it can be derived from the LU decomposition of $\Sigma_n$. For this we define a diagonal matrix $D= (d_{ij})$ using entries from $U$ by
\begin{align*}
d_{ij} &= \left\{
\begin{array}{ll}
0&   \text{ for } \, j\neq i, \\
\frac{V(W)d\kappa_d\prod_{k=1}^{j-1}(x_j-x_k)^2}
{4x_j\prod_{ k = 1}^{j-1} ( x_j + x_k)^2 } &    \text{ for } \, j= i .\\
\end{array}
\right. 
\end{align*}
Then $G$ equals $G=LD^{\frac{1}{2}}$ where $L$ is given as in Theorem \ref{Satz: LR-Zerleg. subkrit. Reg.}. A straightforward calculation concludes the proof.
\end{proof}

\begin{exa} \label{Bsp: Cholesky-Zerleg. subkrit. Reg.}
Consider $\Sigma_2$ for $\tau_1=0$ and $\tau_2 \in \R$. 
Then
\begin{align} \label{Sigma kritisch}
\Sigma_2 = GG^t=
\sqrt{\frac{V(W)d\kappa_d}{2}}\left( \begin{array}{cccc}   
\frac{1}{\sqrt{d}}& 0 \\
\frac{\sqrt{2x_2}}{a_2}& \frac{\tau_2}{\sqrt{2x_2}a_2} \\
\end{array}\right)\sqrt{\frac{V(W)d\kappa_d}{2}}\left( \begin{array}{cccc}   
\frac{1}{\sqrt{d}}&\frac{\sqrt{2x_2}}{a_2} \\
0& \frac{\tau_2}{\sqrt{2x_2}a_2} \\
\end{array}\right).
\end{align}
\end{exa}
 
\section{Critical Regime} \label{Kap: Krit. Reg.}

Concluding the discussions of the last two sections we consider $\Sigma_n$  in the critical regime (i.e. $t\delta_t^d \rightarrow c \in (0, \infty)$) for $n\geq 2$. Recall that 
\begin{align*}
\Sigma_n
=
\begin{cases}
V(W) d\kappa_d \biggl( \frac{a_i a_j+ 2cd\kappa_d (x_i+x_j)}{2(x_i+x_j)a_i a_j} \biggr) 
&\text{ for } c \in (0,\, 1],
\\
V(W) d\kappa_d \biggl( \frac{a_i a_j+ 2cd\kappa_d (x_i+x_j)}{2c(x_i+x_j)a_i a_j} \biggr) 
&\text{ for }  c \in (1, \,\infty).
\end{cases} 
\end{align*}

If matrices from various regimes appear, they are labeled as $\Sigma_n^{\text{sb}}$, $\,\Sigma_n^{\text{sp}}$, and $\Sigma_n^{\text{cr}}$  for the subcritical, supercritical, and critical regime. The abbreviations $x_i = \tau_i +d/2 \, \,\text{ and } \,\, a_i = \tau_i +d \, \text{ for }\, i \in [n]$ used before are taken up again. The  matrices 
$\Sigma_n^{\text{sb}}$ 
by 
(\ref{Eq: Kov.matrix Laengen Potenz Funktional supercritical}) 
and 
$\Sigma_n^{\text{sp}}$ 
by
(\ref{Eq: Kov.matrix Laengen Potenz Funktional subcritical})
using the same data $x_i$ as well as $a_i$
are called \textit{corresponding matrices} of $\Sigma_n^{\text{cr}}$. 
Recall that by Theorem \ref{Satz: Kov.matrix Laengen Potenz Funktional} the matrix $\Sigma_n=\Sigma_n^{\text{cr}}$ is always a sum of $\Sigma_n^{\text{sp}}$ and $\Sigma_n^{\text{sb}}$ using certain coefficients.

The results in the subcritical and supercritical regime yield:
\begin{cor} \label{Satz: pos. Def., Inv., Det. krit. Reg.}
Let $\Sigma_n$ be as in (\ref{Eq: Kov.matrix Laengen Potenz Funktional}) in the critical regime for $n \geq 2$. 
Then 
(i) $\emph{rank}(\Sigma_n) = n$; 
(ii) $\det (\Sigma_n) > 0$;
(iii) $\Sigma_n$ is positive definite. 
\end{cor} 
\begin{proof} 
Theorem \ref{Satz: Kov.matrix Laengen Potenz Funktional} and
Corollaries \ref{Lemma: Courant-Fisher}(iii), \ref{Kor: Rang superkrit. Reg.}(i), \ref{Kor: Det. Kov.matrix subkrit. Reg.}(ii)  yield for $c\in (0,\,1]$ that: 
\[
\text{rank}(\Sigma_n^{\text{cr}}) 
\geq 
\text{max}\{\text{rank}(c\Sigma_n^{\text{sp}}) , \, \text{rank}(\Sigma_n^{\text{sb}})\}= \text{max}\{1, \, n\} = n.
\]
For $c\in (1,\, \infty)$ a similar argument proves the claim.
\end{proof}

In \cite[Proposition 3.4]{Reitzner} the authors proved already the fact above. Additionally, we are able to give formulas for the inverse matrix and determinant of $\Sigma_n^{\text{cr}}$ using the fact that it is a rank-$1$-correction of $\Sigma_n^{\text{sb}}$ as well as the already known results about the adjoint and the inverse matrix in the subcritical regime. 

Indeed, the inverse matrix of $\Sigma_n^{\text{cr}}$ can be computed with the \emph{Woodbury matrix identity} (see, e.g., \cite[0.7.4.2.]{Horn}). It implies the following theorem by observing that
$c\Sigma_n^{\text{sp}} = vv^t$ where $v =\sqrt{c V(W)}d\kappa_d(1/a_1, \, \ldots, \, 1/a_n)^t$, $\Sigma_n^{\text{sp}} = ww^t$ where $w =\sqrt{ V(W)}d\kappa_d(1/a_1, \, \ldots, \, 1/a_n)^t$,
\[
(\Sigma_n^{\text{cr}})^{-1} 
= 
(v \cdot 1 \cdot v^t+\Sigma_n^{\text{sb}})^{-1} 
\text{ for } c \in (0,\, 1],
\text{ and } 
(\Sigma_n^{\text{cr}})^{-1} 
= 
(w \cdot 1 \cdot w^t+1/c\Sigma_n^{\text{sb}})^{-1}
\text{ for } c \in (1,\, \infty).
\] 
\begin{thm} \label{Satz: Inverse krit. Regime}
Let $\Sigma_n^{\text{cr}}$ be as in (\ref{Eq: Kov.matrix Laengen Potenz Funktional}) in the critical regime for $n \geq 2$. For $c \in (0,\, 1]$ we have 
\begin{align*}
    (\Sigma_n^{\text{cr}})^{-1} = (\Sigma_n^{\text{sb}})^{-1} - (\Sigma_n^{\text{sb}})^{-1}v (1+v^t (\Sigma_n^{\text{sb}})^{-1} v)^{-1} v^t (\Sigma_n^{\text{sb}})^{-1} \nonumber
\end{align*}
and for $c \in (1,\, \infty)$ one has
\begin{align*}
    (\Sigma_n^{\text{cr}})^{-1} = (1/c\Sigma_n^{\text{sb}})^{-1} - (1/c\Sigma_n^{\text{sb}})^{-1}w (1+w^t (1/c\Sigma_n^{\text{sb}})^{-1} w)^{-1} w^t (1/c\Sigma_n^{\text{sb}})^{-1} .\nonumber
\end{align*}
\end{thm}

Note that the expressions $(1+v^t (\Sigma_n^{\text{sb}})^{-1} v)^{-1}$ and $(1+w^t (1/c\Sigma_n^{\text{sb}})^{-1} w)^{-1}$ exist since the entries of $(\Sigma_n^{\text{sb}})^{-1}$ multiplied with $1/(a_ia_j)$ are $\neq -1$ according to Proposition \ref{Satz: Inv. subkrit. Reg.}. 

\begin{thm} \label{Satz: Det. krit. Reg.}
Let $\Sigma_n$ be as in (\ref{Eq: Kov.matrix Laengen Potenz Funktional}) in the critical regime with $n \geq 2$. For $c \in (0,\,1 ]$ we have
\begin{align*}
     \emph{det}(\Sigma_n) =D\biggl(\frac{V(W)d\kappa_d}{2}\biggr)^n
\end{align*}
and for $c \in (1,\,\infty )$ there is (within the brackets) an additional $c$ in the denominator. Here 
\begin{align*}
     D&= \biggl(1+\sum_{j=1}^n \frac{cd\kappa_d}{a_j}\biggl(  \frac{4x_j \prod_{k \in  [n] \setminus \{j\}} (x_k+x_j)^2}{a_j\prod_{k \in  [n] \setminus \{j\}} (x_k-x_j)^2}\\
     &+ \sum_{i \in [n] \setminus \{j\}}  -\frac{ 8x_ix_j(x_i+x_j)\prod_{ l \in [n] \setminus \{i, \, j\},\, k \in \{i, \, j\}} (x_k+x_l) }{a_i (x_j-x_i)^2\prod_{l \in  [n]\setminus \{i, \, j\}  ,\, k \in \{i, \, j\}} (x_k-x_l)}\biggr)\biggr) \frac{\prod_{1 \leq i < j \leq n} (x_i -x_j)^2}{ \prod_{ 1 \leq i, \,j \leq n }( x_i +x_j ) }.
\end{align*}
\end{thm}

\begin{proof}
With the notations for $v$ and $w$ introduced before Theorem \ref{Satz: Inverse krit. Regime} it follows from the matrix determinant lemma
(also called Cauchy's formula, see \cite[0.8.5.11.]{Horn}) using the inverse of the already known concentration matrix $\Sigma_n^{\text{sb}}$ for $c \in (0,\,1 ]$ that
\begin{align*}
     \text{det}(\Sigma_n^{\text{cr}}) &=\text{det}(\Sigma_n^{\text{sb}} + c\Sigma_n^{\text{sp}}  ) =  (1+ v^t (\Sigma_n^{\text{sb}})^{-1}v)\text{det}(\Sigma_n^{\text{sb}}) 
\end{align*}
 and for $c \in (1,\, \infty)$ that
\begin{align*}
     \text{det}(\Sigma_n^{\text{cr}}) &=\text{det}(1/c\Sigma_n^{\text{sb}} + \Sigma_n^{\text{sp}}  )
     =  (1+ w^t (1/c\Sigma_n^{\text{sb}})^{-1}w)(1/c)^n\text{det}(\Sigma_n^{\text{sb}}).
\end{align*}
Next we insert the determinant in the subcritical regime from Corollary \ref{Kor: Det. Kov.matrix subkrit. Reg.} and compute $v^t (\Sigma_n^{\text{sb}})^{-1}v$ in the case $c \in (0,\,1 ]$ and $w^t (1/c\Sigma_n^{\text{sb}})^{-1}w$ in the case $c \in (1,\,\infty )$. 
In the first case Proposition \ref{Satz: Inv. subkrit. Reg.} implies that the entry in place $j$ of the 
$1 \times n$-matrix $v^t (\Sigma_n^{\text{sb}})^{-1}$ is
\begin{align*}
  &\left( \frac{\sqrt{c \cdot V(W)}d\kappa_d4x_j \prod_{k \in  [n] \setminus \{j\}} (x_k+x_j)^2}{a_j V(W)d\kappa_d\prod_{k \in  [n] \setminus \{j\}} (x_k-x_j)^2}\right)\\
  &+ \sum_{i \in [n] \setminus \{j\}} \left( -\frac{ \sqrt{c \cdot V(W)}d\kappa_d8x_ix_j(x_i+x_j)\prod_{ l \in [n] \setminus \{i, \, j\},\, k \in \{i, \, j\}} (x_k+x_l) }{a_i V(W)d\kappa_d(x_j-x_i)^2\prod_{l \in  [n]\setminus \{i, \, j\}  ,\, k \in \{i, \, j\}} (x_k-x_l)}\right)  .
\end{align*}
A multiplication of this row vector with $v$ from the right yields
\begin{align*}
  &\sum_{j=1}^n \frac{cd\kappa_d}{a_j}\biggl(  \frac{4x_j \prod_{k \in  [n] \setminus \{j\}} (x_k+x_j)^2}{a_j\prod_{k \in  [n] \setminus \{j\}} (x_k-x_j)^2}+ \sum_{i \in [n] \setminus \{j\}}  -\frac{ 8x_ix_j(x_i+x_j)\prod_{ l \in [n] \setminus \{i, \, j\},\, k \in \{i, \, j\}} (x_k+x_l) }{a_i (x_j-x_i)^2\prod_{l \in  [n]\setminus \{i, \, j\}  ,\, k \in \{i, \, j\}} (x_k-x_l)}\biggr).
\end{align*}
For $c \in (1,\, \infty)$ an analogue computation concludes the proof.
\end{proof}

As a next goal we are discussing questions related to eigenvalues of $\Sigma_n$ in the critical regime. 
\begin{exa} \label{Bsp: EWe krit. Reg.}
Consider for $n=2$ and $c \in (0,\, 1]$ the covariance matrix 
\begin{align*}
\Sigma_2=V(W)\left( \begin{array}{cccccc}
\frac{d\kappa_d(a_1^2+ 2cd\kappa_d 2x_1)}{4x_1a_1^2} &\frac{d\kappa_d(a_1a_2+ 2cd\kappa_d (x_1+x_2))}{2(x_1+x_2)a_1a_2}\\
 \frac{d\kappa_d(a_1a_2+ 2cd\kappa_d (x_1+x_2))}{2(x_1+x_2)a_1a_2} & \frac{d\kappa_d(a_2^2+ 2cd\kappa_d 2x_2)}{4x_2a_2^2}\\
\end{array}\right).
\end{align*}
The case where $d=2$, $V(W)=1$ and (non-negative increasing powers) $\tau_1=0, \, \tau_2 = 1$ is of special interest. 
In particular, we consider  $c=1$. Then a calculation of the characteristic polynomial with eigenvalues $\lambda_1$, $\lambda_2$ of $\Sigma_2$ leads to
\begin{align*}
\chi(\Sigma_2) &= \lambda^2 - \lambda \cdot\text{tr}(\Sigma_2) +\text{det}(\Sigma_2)
=\lambda^2 - \lambda\left( \frac{13 \pi^2}{9}+\frac{3\pi}{4}\right) + \frac{\pi^3}{36}+ \frac{\pi^2}{72} \in \R[\lambda].
\end{align*}
We obtain
$\lambda_{1,2} 
= \frac{\pi}{72}\left(52\pi+27\mp\sqrt{52^2\pi^2+2664\pi+657}\right).$
\end{exa}

The characteristic polynomial can be determined for the critical regime as in the other possible regimes. Lemma \ref{Satz: char. Poly. OmegaRS subkrit. Reg.} and Theorem \ref{Satz: Det. krit. Reg.} yield the following: 

\begin{cor} \label{Satz: char. Poly. Sigma krit. Reg.}
Let $\Sigma_n = (\sigma_{ij})$ be as in (\ref{Eq: Kov.matrix Laengen Potenz Funktional}) in the critical regime for $n \geq 2$ with characteristic polynomial
$\chi(\Sigma_n) = (-1)^n \lambda^n+ a_{n}^{(n-1)}\lambda^{n-1}+ \ldots  + a_n^{(1)} \lambda+ a_n^{(0)} \in \mathbb{R}[\lambda]$. 
For $c \in (0,\,1]$ we have 
\begin{align*}
  a^{(k)}_{n} &= (-1)^k  D_k  \biggl(\frac{V(W)d\kappa_d}{2}\biggr)^{n-k} \text{ for } \, k \in [n-1] .  
\end{align*}
For $c \in (1,\,\infty )$ there is an additional $c$ in the denominator of $V(W)d\kappa_d/2$. Here
\begin{align*}
   D_k= &\sum \limits_{1 \leq i_1< \ldots< i_{n-k} \leq n }\biggl(1+\sum_{j \in\{i_1,\, \ldots, \,i_{n-k}\}} \frac{cd\kappa_d}{a_j}\biggl(  \frac{4x_j \prod_{k \in  \{i_1,\, \ldots, \,i_{n-k}\} \setminus \{j\}} (x_k+x_j)^2}{a_j\prod_{k \in \{i_1,\, \ldots, \,i_{n-k}\} \setminus \{j\}} (x_k-x_j)^2}\\
     &+ \sum_{i \in \{i_1,\, \ldots, \,i_{n-k}\} \setminus \{j\}}  -\frac{ 8x_ix_j(x_i+x_j)\prod_{ l \in \{i_1,\, \ldots, \,i_{n-k}\} \setminus \{i, \, j\},\, k \in \{i, \, j\}} (x_k+x_l) }{a_i (x_j-x_i)^2\prod_{l \in  \{i_1,\, \ldots, \,i_{n-k}\}\setminus \{i, \, j\}  ,\, k \in \{i, \, j\}} (x_k-x_l)}\biggr)\biggr)  \\
     & \frac{\prod_{i < j \in \{i_1,\, \ldots, \,i_{n-k}\}} (x_i -x_j)^2}{ \prod_{  i, \,j \in \{i_1,\, \ldots, \,i_{n-k}\} }( x_i +x_j ) }  
\end{align*}
for $k \in [n-1]$. 
\end{cor}

As in the subcritical regime we can not determine eigenvalues explicitly. 
Upper and lower bounds for them are given below. See also, e.g., \cite{Ding} for related results using other methods.

\begin{thm} \label{Satz: Schranken EWe krit. Reg.}
Let $\Sigma_n^{\text{cr}} $ be as in (\ref{Eq: Kov.matrix Laengen Potenz Funktional}) in the critical regime for $n \geq 2$ with eigenvalues $\lambda_1^{\text{cr}} \leq \ldots \leq \lambda_n^{\text{cr}} $ and let $\bar{S}_n$ and $\b{S}_n$ be as in Theorem \ref{Satz: Schranken EWe Sigma subkrit. Reg.}. We have
\[
\textit{\b{U}}_n \leq \lambda_1^{\text{cr}} \leq  \ldots \leq \lambda_n^{\text{cr}} \leq \bar{U}_n
\text{ where }
\]
\begin{enumerate}
\item $
\textit{\b{U}}_n = \textit{\b{S}}_n  \, \text{ and } \,\bar{U}_n = \bar{S}_n  + \sum_{i=1}^n \frac{cV(W)d^2\kappa_d^2}{ a_i^2 } \,\, \text{ for } c \in (0,\, 1],$
\item $\b{U}_n = \frac{\textit{\b{S}}_n }{c}\, \text{ and }\, \bar{U}_n =\frac{\bar{S}_n }{c}+\sum_{i=1}^n \frac{V(W)d^2\kappa_d^2}{ a_i^2 }\,\, \text{ for }c \in (1,\, \infty).$
\end{enumerate}
\end{thm} 
\begin{proof}
Consider $ c \in (0,\,1]$. Let $\lambda^{t}_1 \leq \lambda_2 ^{t}\leq\ldots \leq \lambda_n^{t} $ for $t\in \{\text{sb}, \,\text{sp}\}$ be all eigenvalues of the corresponding matrices $\Sigma_n^{\text{sb}}$ and $\Sigma_n^{\text{sp}}$ of the matrix $\Sigma_n^{\text{cr}}$. Due to Theorem \ref{Satz: EW.e, ER.e superkrit. Reg.} the eigenvalues of $\Sigma_n^{\text{sp}}$ are $\lambda_1^{\text{sp}} = \lambda_2^{\text{sp}} = \ldots = \lambda_{n-1} ^{\text{sp}}=0\,\text{ and } \,\lambda_n ^{\text{sp}}= V(W)d^2\kappa_d^2(\sum_{i=1}^n 1/ a_i^2 ).$ Theorem \ref{Satz: Schranken EWe Sigma subkrit. Reg.} provides an upper bound $\bar{S}_n$ and a lower bound $\emph{\b{S}}_n$ for the eigenvalues $\lambda_1^{\text{sb}}\leq \ldots \leq\lambda_n^{\text{sb}}$ of $\Sigma_n^{\text{sb}}$. Using Lemma \ref{Lemma: Courant-Fisher}(i), one can state bounds for the eigenvalues of $\Sigma_n^{\text{cr}}$. Note that the multiplicative constant $c$ occurring in the sum decomposition
of $\Sigma_n^{\text{cr}}$
must be taken into account. Thus,
\begin{align*} 
\emph{\b{S}}_n+c\cdot 0 \leq \lambda_1^{\text{sb}} + c \lambda_1^{\text{sp}} \overset{\ref{Lemma: Courant-Fisher}(i)}{\leq} \lambda_1^{\text{cr}} \leq \ldots \leq \lambda_n^{\text{cr}} \overset{\ref{Lemma: Courant-Fisher}(i)}{\leq} \lambda_n^{\text{sb}} + c \lambda_n^{\text{sp}} \leq \bar{S}_n + c\sum_{i=1}^n \frac{V(W)d^2 \kappa_d^2}{a_i^2}.
\end{align*}
The proof of the case $c \in (1,\, \infty)$ is done in an analogue way.
\end{proof} 

\begin{rem} \label{Bem: Schranken EWe krit. Reg.}
It is of interest to improve the bounds from the previous theorem. A possible approach for this is to use eigenvalues $\lambda_i^{\text{sb}}$ of $\Sigma_n^{\text{sb}}$ themselves.
Then Lemma \ref{Lemma: Courant-Fisher}(i) yields:
\begin{enumerate}   
 \item If $c \in (0,\, 1]$, then \,$ \lambda_i^{\text{sb}}  \leq \lambda_i^{\text{cr}} \leq \lambda^{\text{sb}}_i+ c V(W)d^2\kappa_d^2 (\sum_{i=1}^n \frac{1}{a_i^2 } )$ for $i \in [n]$.
 \item If $ c \in (1,\, \infty)$, then \,$ \frac{1}{c}\lambda^{\text{sb}}_i\leq \lambda_i^{\text{cr}} \leq \frac{1}{c}\lambda^{\text{sb}}_i+V(W)d^2\kappa_d^2(\sum_{i=1}^n \frac{1}{ a_i^2 } )$ for $i \in [n]$.
\end{enumerate}
\end{rem} 

Note that this remark yields at the moment an improvement only for $n=2$, since the eigenvalues $\lambda_{n}^{\text{sb}}$ for $n>2$ have not been explicitly determined yet.

\begin{exa} \label{Bsp: Schranken EWe krit. Reg.}   
Consider $\Sigma_2^{\text{cr}}$ from Example \ref{Bsp: EWe krit. Reg.} with $c =1, \, d=2, V(W)=1$ and (non-negative increasing  powers) $\tau_1 =0$ and $\tau_2 =1$. The bounds from Theorem \ref{Satz: Schranken EWe krit. Reg.} lead to
\[\emph{\b{U}}_2= \frac{\pi}{24}(9-\sqrt{73})  \leq \lambda_1^{\text{cr}} \leq \lambda_2^{\text{cr}} \leq \frac{\pi}{24}(9+\sqrt{73})  + \frac{13\pi^2}{9} =\bar{U}_2.\]
Observe that Remark \ref{Bem: Schranken EWe krit. Reg.} is applicable. 
By Example \ref{Bsp: EWe krit. Reg.} the eigenvalues of $\Sigma^{\text{sb}}_2$ are
$\lambda^{\text{sb}}_1 =\frac{\pi}{24}(9-\sqrt{73}) \, \text{ and } \,
\lambda^{\text{sb}}_2  =\frac{\pi}{24}(9+\sqrt{73}) .$
According to Theorem \ref{Satz: EW.e, ER.e superkrit. Reg.} all eigenvalues of $\Sigma_2^{\text{sp}}$ are given as $\lambda^{\text{sp}}_1 =0  \,\text{ and } \, \lambda^{\text{sp}}_2 =13\pi^2/9.$ Hence,
\begin{align*}
&
\frac{\pi}{24}(9-\sqrt{73})  \leq \lambda_1^{\text{cr}} \leq \frac{\pi}{24}(9-\sqrt{73})  + \frac{13\pi^2}{9} 
\\
& \frac{\pi}{24}(9+\sqrt{73})  \leq \lambda_2^{\text{cr}} \leq \frac{\pi}{24}(9+\sqrt{73})  + \frac{13\pi^2}{9}.
\end{align*}
Note that according to these bounds the eigenvalues 
are potentially lying in an interval of length
$\lambda_2^{\text{sp}} =\frac{13\pi^2}{9}< 2\cdot 10^2$. 
However,  we know that
\[
\lambda_{1, \, 2}^{\text{cr}}= \frac{\pi}{72}\left(52\pi+27\mp\sqrt{52^2\pi^2+2664\pi+657}\right)
\]
and these numbers are close to the end points of that interval, since they
have only a distance $\leq 10^{-3}$ to $\bar{U}_2$ or $ \emph{\b{U}}_2$. 
This examples motivates to study improvements for $\bar{U}_n$ and $\emph{\b{U}}_n$ to obtain better bounds.
\end{exa} 

We state the following conjecture regarding eigenvalues of the considered matrices.

\begin{conj} \label{Verm: EWe, Eigenraeume, Diag.matrix krit. Reg.}
Let $\Sigma_n$ be as in (\ref{Eq: Kov.matrix Laengen Potenz Funktional}) in the critical regime for $n \geq 2$ with eigenvalues $\lambda_1 \leq \ldots \leq \lambda_n$. Then $\lambda_i \neq \lambda_j $ for all  $i \neq j \in [n]$.
\end{conj}

Finally, we investigate decompositions of $\Sigma_n$ in the critical regime. They become very complicated in this regime. That's why we focus here on an interesting special case:

\begin{Def}
Let $\tau_i = i-1$ for $i=1,\dots,n$ and $d=2$. The vector of length power functionals $(\tilde{L}_t^{(\tau_1)}, \ldots, \,\tilde{L}_t^{(\tau_n)})$ according to (\ref{Eq: normiertes, zentriertes Laengen Potenz Funktional}) is called \textit{vector of natural increasing powers}.
\end{Def}

For the vector of natural increasing powers the LU decomposition is:

\begin{prop}
\label{Satz: LR-Zerlegung krit. Reg.}
Let $\Sigma_n$ be as in (\ref{Eq: Kov.matrix Laengen Potenz Funktional}) in the critical regime with $n \geq 2$ with respect to the vector of natural increasing powers. The LU decomposition $\Sigma_n = L U$ with $L =  (l_{ij})$ and $ U = (u_{ij})$ is given by 
\begin{align*}
l_{ij} 
&= 
\left\{
\begin{array}{ll}
0& \,  \text{ for } \, i < j, \\
\frac{\prod_{k=1}^{j-1}(i-k)\prod_{k=1}^{j} (j+k)}{\prod_{k=1}^{j-1} (j-k)\prod_{ k = 1}^{j} ( i + k)}& \, \text{ for }\, i>j,\\
\end{array}
\right. 
\\
\text{ and }\,
u_{ij} 
&= 
\left\{
\begin{array}{ll}
\frac{ \pi (1+ 2c\pi )}{j+1}  & \, \text{ for }\, i=1 \text{ and }  c \in (0, \, 1],\\ 
\frac{ \pi (1+2c\pi)}{c(j+1)} & \, \text{ for }\, i=1 \text{ and } c \in  (1, \, \infty),\\ 
0 & \,  \text{ for } \, i > j,\\
\frac{\pi\prod_{k=1}^{i-1}(i-k)(j-k)}{(i+j)\prod_{ k = 1}^{i-1} ( i + k) \prod_{k=1}^{i-1}(j + k)} & \, \text{ for }\, 1<i \leq j \text{ and }c \in (0,\, 1],\\
\frac{\pi\prod_{k=1}^{i-1}(i-k)(j-k)}{c(i+j)\prod_{ k = 1}^{i-1} ( i + k) \prod_{k=1}^{i-1}(j + k)} & \, \text{ for }\, 1<i \leq j \text{ and } c \in (1, \, \infty). \\
\end{array}
\right. 
\end{align*}
\end{prop}
 \begin{proof}
We compute the entries $\sigma_{ij} = \sum_{m=1} ^n l_{im}u_{mj}$ of the given matrix $LU$ for $c \in (0, \, 1]$. For this let $q = \min \{i,j\}$. One sees $ l_{im}u_{mj} = 0$ for $m \in \{ q+1, \ldots , n\}$ and $l_{i1}u_{1j} =2\pi(1+2c\pi)/(i+1)(j+1)$. Thus, 
\begin{align*}
    \sigma_{ij} &=
    \frac{2\pi(1+2c\pi)}{(i+1)(j+1)}+\sum_{m=2} ^q  \frac{\prod_{k=1}^{m-1}(i-k)\prod_{k=1}^{m} (m+k)}{\prod_{k=1}^{m-1} (m-k)\prod_{ k = 1}^{m} ( i + k)}\cdot\frac{\pi\prod_{k=1}^{m-1}(m-k)(j-k)}{(m+j)\prod_{ k = 1}^{m-1} (m+ k) \prod_{k=1}^{m-1}(j + k)} 
    \\
    &=  \frac{2\pi(1+2c\pi)}{(i+1)(j+1)}+\sum_{m=2} ^q  \frac{2\pi m\prod_{k=1}^{m-1}(i-k)(j-k)}{\prod_{ k = 1}^{m} ( i + k)(j + k)}
    \\
    &=  \frac{2\pi(1+2c\pi)}{(i+1)(j+1)}+\sum_{m=1} ^q  \frac{2\pi m\prod_{k=1}^{m-1}(i-k)(j-k)\prod_{k=m+1}^q(i+1)(j+1)}{\prod_{ k = 1}^{q} ( i + k)(j + k)}- \frac{2\pi}{(i+1)(j+1)}\\
       &\overset{(\ref{Eq: Formal Produkt})}{=}  \frac{2\pi(1+2c\pi)}{(i+1)(j+1)}+\frac{\pi}{(i+j)}- \frac{2\pi}{(i+1)(j+1)} = \frac{\pi( 4c\pi(i+j) +(i+1)(j+1) )}{(i+j)(i+1)(j+1)}.
\end{align*}
These are the entries of $\Sigma_n^{\text{cr}}$ as in (\ref{Sigma kritisch}) for the vector of natural increasing powers.
The case for $c \in (1, \, \infty)$ is treated analogously.
\end{proof}

The latter results implies:

\begin{cor}\label{Satz: Determ. nat. Pot. krit. Reg.}
Let $\Sigma_n$ be as in (\ref{Eq: Kov.matrix Laengen Potenz Funktional}) in the critical regime with $n \geq 2$ with respect to the vector of natural increasing powers. For $c \in (0,\, 1]$ the determinant of $\Sigma_n$ is given by 
\[ \emph{det}(\Sigma_n ) = \frac{\pi (1+2c\pi)}{2} \prod_{i=2}^n \frac{\pi\prod_{k=1}^{i-1}(i-k)^2}{2i\prod_{ k = 1}^{i-1} ( i + k)^2}. \] 
For $c \in (1, \, \infty)$ there is an additional factor $c^n$ in the denominator of $\pi(1+2c\pi)/2$.
\end{cor}

The LU decomposition of $\Sigma_n$ with respect to the vector of natural increasing powers yields immediately the Cholesky decomposition of $\Sigma_n$: 

\begin{cor} \label{Satz: Cholesky-Zerleg. krit. Reg.}
Let $\Sigma_n$ be as in (\ref{Eq: Kov.matrix Laengen Potenz Funktional}) in the critical regime for $n \geq 2$ with respect to the vector of natural increasing powers. The Cholesky decomposition $ \Sigma_n = G G^t$ is given by $G = (g_{ij})$, such that entries for $c \in (0,1]$ are given by 
\begin{align*} 
g_{ij} = \left\{
\begin{array}{ll}
0& \text{ for } \, j>i, \\
\frac{ 2}{  i + 1}\sqrt{\frac{ \pi (1+ 2c\pi)}{2}}  & \text{ for } \,j=1 \text{ and }  i\geq 1,  
\\ 
\frac{\prod_{k=1}^{j-1}(i-k)\prod_{k=1}^{j} (j+k)}{\prod_{k=1}^{j-1} (j-k)\prod_{ k = 1}^{j} ( i + k)}\sqrt{\frac{\pi\prod_{k=1}^{j-1}(j-k)^2}{2j\prod_{k=1}^{j-1}(j+k)^2}}&  \text{ for }\, 1 <j\leq i.\\
\end{array}
\right.
\end{align*}
For $c \in (1, \,\infty)$ there is an additional $c$ in every denominator of all occurring roots.
\end{cor}
   
\begin{proof}
The unique Cholesky decomposition of $\Sigma_n$ exists due to Algorithm \ref{Algo: Korrektheit Cholesky Alg. pos. semidef.} and it can be derived from the LU decomposition $\Sigma_n = LU$ determined in Proposition \ref{Satz: LR-Zerlegung krit. Reg.}. 

For this, define the diagonal matrix $D= (d_{ij})$ with diagonal entries from $U$ given by
\begin{align*}
d_{ij} &= \left\{
\begin{array}{ll}
0&   \text{ for } \, j\neq i, \\
\frac{\pi(1+2c\pi)}{(j+1)} & \text{ for }\, j=i=1 \text{ and } c \in (0,\,1], \\
\frac{\pi(1+2c\pi)}{c(j+1)}, & \text{ for }\, j=i=1  \text{ and }  c \in (1,\, \infty), \\
\frac{\pi\prod_{k=1}^{j-1}(j-k)^2}{2j\prod_{k=1}^{j-1}(j+k)^2}& \text{ for } \,j=i>1 \text{ and } c \in (0,\,1],\\
\frac{\pi\prod_{k=1}^{j-1}(j-k)^2}{2cj\prod_{k=1}^{j-1}(j+k)^2}  &\text{ for }\, j=i>1  \text{ and } c \in (1, \, \infty).\\
\end{array}
\right. 
\end{align*}

This yields $G=LD^{\frac{1}{2}}$. This concludes the proof.
\end{proof}

Finally, we compute an example of a Cholesky decomposition of $\Sigma_2$ for arbitrary powers, which will be used in Section \ref{Kap: stoch. appl.}.

\begin{exa} \label{Bsp.Cholesky krit. Reg.}
Consider $\Sigma_2$ in the critical regime. For $c \in (0,1]$ we see
\[\Sigma_2= \frac{V(W)d\kappa_d}{2}\left( \begin{array}{cccc}   
\frac{4d\kappa_dcx_1+a_1^2}{2x_1a_1^2}& \frac{2d\kappa_dc(x_1+x_2)+a_1a_2}{(x_1+x_2)a_1a_2} \\
\frac{2d\kappa_dc(x_1+x_2)+a_1a_2}{(x_1+x_2)a_1a_2}& \frac{4d\kappa_dcx_2+a_2^2}{2x_2a_2^2} \\
\end{array}\right)\] and therefore 
\[
\Sigma_2 = GG^t
\]
with the following matrix $G$:
\begin{align*}
 \sqrt{ \frac{V(W)d\kappa_d}{2}}\left( \begin{array}{cccc}   
\sqrt{\frac{4d\kappa_dcx_1+a_1^2}{2x_1a_1^2}}& 0 \\
\frac{(2d\kappa_dc(x_1+x_2)+a_1a_2)\sqrt{2x_1}}{a_2(x_1+x_2)\sqrt{4d\kappa_dcx_1+a_1^2}}& \sqrt{\frac{4d\kappa_dcx_2+a_2^2}{2x_2a_2^2}-\biggl(\frac{(2d\kappa_dc(x_1+x_2)+a_1a_2)\sqrt{2x_1}}{a_2(x_1+x_2)\sqrt{4d\kappa_dcx_1+a_1^2}}\biggr)^2}  \\
\end{array}\right). \\
\end{align*}
For $\tau_1 = 0 $ and $\tau_2\in \R$ we get for instance
\begin{align*}
G&=  \sqrt{\frac{V(W)d\kappa_d}{2}}\left( \begin{array}{cccc}   
\sqrt{\frac{2\kappa_d c+1}{d}}& 0 \\
\frac{\sqrt{(2\kappa_dc+1)d}}{a_2}& \sqrt{\frac{\tau^2}{2x_2a_2^2}} \\
\end{array}\right).
\end{align*}
For $c \in (1,\,\infty)$ there is an additional $\sqrt{c}$ in each denominator of the entries of the matrix $G$.
\end{exa}
\section{Across-regimes investigations} \label{Kap. Regimeunabhaengige Untersuchung II}
In this Section  some across regimes results regarding decompositions are presented. As before, we use  labels $\text{cr}, \, \text{sb}$ or $\text{sp}$ as well as notations $l_{ij}, u_{ij}$ and $g_{ij}$ for the entries of corresponding matrices in decompositions of interest. 

\begin{thm} \label{Satz: LR-Zerleg. krit. Reg. Zsmh. super. und subkrit. Reg}
Let $\Sigma_n^{\text{cr}}$ be as in (\ref{Eq: Kov.matrix Laengen Potenz Funktional}) in the critical regime for $n \geq 2$ with respect to the vector of natural increasing powers. 
Consider the LU decompositions 
$\Sigma_n^{\text{sb}}=L^{\text{sb}} U^{\text{sb}}$
and 
$\Sigma_n^{\text{sp}}=L^{\text{sp}}U^{\text{sp}}$.
Then the entries of the matrices $L^{\text{cr}},\,U^{\text{cr}}$ 
of the LU decomposition of $\Sigma_n^{\text{cr}}$ are given as:
\begin{enumerate}
\item For $c \in (0,\, 1]$  we have $l_{ij}^{\text{cr}} =l^{\text{sb}}_{ij} \,\text{ for } \,i,\,j \in [n] \,\text{ and } \, u_{ij}^{\text{cr}} = u_{ij}^{\text{sb}} +cu_{ij}^{\text{sp}}.$
\item For $ c \in (1, \, \infty)$  we have
$l_{ij}^{\text{cr}} =l^{\text{sb}}_{ij} \,\text{ for }\,i,\,j \in [n] \,\text{ and } \, u_{ij}^{\text{cr}} = \frac{1}{c} u_{ij}^{\text{sb}} + u_{ij}^{\text{sp}} .$
\end{enumerate}
\end{thm}

\begin{proof}
Let $c \in (0,\, 1]$. Theorem \ref{Satz: Kov.matrix Laengen Potenz Funktional} yields 
\[
(L^{\text{sp}})^{-1} \Sigma_n^{\text{cr}} = (L^{\text{sp}})^{-1} \Sigma_n^{\text{sb}} + (L^{\text{sp}})^{-1} c\Sigma_n^{\text{sp}}.
\]
Since the first column 
of the matrices $L^{\text{sp}}$ and $L^{\text{sb}}$ coincide, it follows
\begin{equation}
\label{eq-LU-crit}
(L^{\text{sp}})^{-1} \Sigma_n^{\text{cr}} = (U^{(1)})^{\text{sb}} +  cU^{\text{sp}}.
\end{equation}
Here on the right hand, 
the multiplication with $(L^{\text{sp}})^{-1}$ corresponds to the use of
the first (Frobenius) matrix in an LU algorithm (see Theorem \ref{Alg: LR}).
For the first summand $(U^{(1)})^{\text{sb}} $ denotes the resulting matrix by multiplying with $\Sigma_n^{\text{sb}}$. For the second summand we have $U^{\text{sp}}=(L^{\text{sp}})^{-1} \Sigma_n^{\text{sp}}$.

Indeed,  according to Theorem \ref{Satz: LR-Zerleg. superkrit. Reg.}(i) the matrix $U^{\text{sp}}$ has only zeros below the first row, so transformations in the algorithm using left multiplications of further (Frobenius) matrices (which are zero in the places $(i,j)$ with $i>j=1$) do not change anything. 

We apply the remaining $ n-2 $ steps of the LU algorithm determining the LU decomposition of $\Sigma^{\text{sb}}_n $ to both sides of the Equation (\ref{eq-LU-crit}), i.e.~multiplications from the left with suitable Frobenius matrices to transform $ \Sigma ^ { \text{sb}}_n $ into an upper triangle matrix. Multiplying all these lower triangular matrices yields exactly $(L^{\text{sb}})^{-1}$. Hence,
\[
(L^{\text{sb}})^{-1} \Sigma_n^{\text{cr}} 
= 
(L^{\text{sb}})^{-1} \Sigma_n^{\text{sb}} + cU^{\text{sp}}
= 
U^{\text{sb}} + cU^{\text{sp}}.
\]
The right side of the latter equation is an upper triangular matrix and the left side is a lower triangular matrix multiplied with $\Sigma_n^{\text{cr}}$. This leads to the LU decomposition of $\Sigma_n^{\text{cr}}$: 
\[ \Sigma_n^{\text{cr}} =L^{\text{sb}} ( U^{\text{sb}} + cU^{\text{sp}} ).\]
For $c \in (1, \,\infty)$ one uses the same transformations and sees
$ \Sigma_n^{\text{cr}} =L^{\text{sb}} ( \frac{1}{c} U^{\text{sb}} + U^{\text{sp}} ).$ 
\end{proof}

\begin{rem} 
\label{Bem: LR Zsmhg. krit., subkrit.} 
Since $cu_{1j}^{\text{sp}} =  2c\pi\frac{\pi}{j+1}=2c\pi u_{1j}^{\text{sb}} \text{ and } u_{1j}^{\text{sp}} = \frac{2\pi^2}{j+1} = 2c\pi \frac{u_{1j}^{\text{sb}}}{c}$, all entries of $L^{\text{cr}}$ and $U^{\text{cr}}$ can be computed using only the entries of $L^{\text{sb}}$ and $U^{\text{sb}}$. Thus, for $ i,\, j \in [n]$ we have
\begin{align*}
l_{ij}^{\text{cr}} &= l_{ij}^{\text{sb}} \, \text{ and } \, u_{ij}^{\text{cr}} = m u_{ij}^{\text{sb}}\,  
\text{ with }\, m=\left\{
\begin{array}{ll}
1+ 2c\pi &  \text{ for }  \,c \in (0, \, 1], \\ 
\frac{1+2c\pi}{c}  & \text{ for } \,c \in  (1, \, \infty).\\ 
\end{array} \right.
\end{align*}

\end{rem}

It is easy to see that Theorem \ref{Satz: LR-Zerleg. krit. Reg. Zsmh. super. und subkrit. Reg} can not be generalized to an arbitrary vector of length power functionals. 
For example, the choices $\tau_1 =1, \, \tau_2=3, \, \tau_3=5$, $d = 2$, and $c = 1 $
leads to a counterexample. For Cholesky decompositions we get:

\begin{thm} \label{Satz: Cholesky-Zerleg. krit. Reg. Zsmhg. superkrit und krit. Reg.}
Let $\Sigma_n^{\text{cr}}$ be as in (\ref{Eq: Kov.matrix Laengen Potenz Funktional}) in the critical regime for $n \geq 2$ to the vector of natural increasing powers. Then for Cholesky factors $G^{\emph{cr}}$ of $\Sigma_n^{\text{cr}}$ we have  $g_{ij}^{\emph{cr}} = g^{\emph{sb}}_{ij} \, \text{ for } \,j>1 \text{ and }  g_{i1}^{\emph{cr}} = v g^{\emph{sb}}_{i1}$
with
\[v =\left\{
\begin{array}{ll}
 \sqrt{1+\frac{c u_{11}^{\emph{sp}}}{ u_{11}^{\emph{sb}}}} & \emph{ for } \, c \in (0, \, 1], \\ 
\sqrt{\frac{1}{c}+\frac{ u_{11}^{\emph{sp}}}{u_{11}^{\emph{sb}}}}  &\emph{ for }\, c \in  (1, \, \infty).\\ 
\end{array} \right.\]
\end{thm}

\begin{proof} 
A comparison of Corollaries \ref{Satz: Cholesky-Zerleg. subkrit. Reg.} and \ref{Satz: Cholesky-Zerleg. krit. Reg.} shows that $g_{ij}^{\text{cr}} $ and $g_{ij}^{\text{sb}}$ coincide in all entries, except for those in the first column. The factor by which $g_{i1}^{\text{cr}}$ and $g_{i1}^{\text{sb}}$ differ, arises by the use of a diagonal matrix similar to the one used in the proof of Corollary \ref{Satz: Cholesky-Zerleg. krit. Reg.}. Indeed, the matrices $L^{\text{cr}}$ and $L^{\text{sb}}$ are identical. The same is true for the matrices $U^{\text{cr}}$ and $U^{\text{sb}}$, except for their first rows. Here we have  $u_{1j}^{\text{cr}}= u_{1j}^{\text{sb}} + c u_{1j}^{\text{sp}}$ for $ c \in (0,\, 1]$ and  $u_{1j}^{\text{cr}}= \frac{1}{c}u_{1j}^{\text{sb}} +  u_{1j}^{\text{sp}}$ for $ c \in (1,\, \infty)$ according to Theorem \ref{Satz: LR-Zerleg. krit. Reg. Zsmh. super. und subkrit. Reg}. Set
\begin{align*}
m=\frac{u_{11}^{\text{cr}}}{u_{11}^{\text{sb}}} = \left\{
\begin{array}{ll}
 1+\frac{c u_{11}^{\text{sp}}}{u_{11}^{\text{sb}}}   & \text{ for }  c \in (0, \, 1], \\ 
\frac{1}{c}+\frac{u_{11}^{\text{sp}}}{u_{11}^{\text{sb}}}  &  \text{ for } c \in  (1, \, \infty).\\ 
\end{array} \right.
\end{align*} 
Observe that the entries $d_{11}^{\text{cr}}$ of $D^{\text{cr}}$ and $d^{\text{sb}}_{11}$ of $D^{\text{sb}}$ from the proof of Corollary \ref{Satz: Cholesky-Zerleg. krit. Reg.} are related as $d_{11}^{\text{cr}}=m d_{11}^{\text{sb}}$. Hence, 
\[ g_{i1}^{\text{sb}} =l_{i1}^{\text{sb}} \sqrt{d_{11}^{\text{sb}} } \,\text{ and }\,  g_{i1}^{\text{cr}} =l_{i1}^{\text{cr}}  \sqrt{d_{11}^{\text{cr}} }=l_{i1}^{\text{sb}} \sqrt{m d_{11}^{\text{sb}} }.\]
Thus, setting $ v = \sqrt{m}$ one obtains $g_{i1}^{\text{cr}} = v  g_{i1}^{\text{sb}}.$ 
\end{proof}

Regarding the Cholesky decomposition there arise observations similar to Remark \ref{Bem: LR Zsmhg. krit., subkrit.}:

\begin{rem}\label{Bem: Cholesky Zsmhg. krit., subkrit.} 
It is possible to give a formula for $G^{\text{cr}}$ with respect to the vector of natural increasing powers without using $U^{\text{sp}}$ (but still $U^{\text{sb}}$) in contrast to the one determined in Theorem \ref{Satz: Cholesky-Zerleg. krit. Reg. Zsmhg. superkrit und krit. Reg.}.
For this observe at first that comparing the entries $u_{11}^{\text{sb}}$ and $u_{11}^{\text{sp}}$ one gets $u_{11}^{\text{sp}} = 2\pi u_{11}^{\text{sb}}$. Using the notation $m$ and $v$ as in the proof of Theorem \ref{Satz: Cholesky-Zerleg. krit. Reg. Zsmhg. superkrit und krit. Reg.} this leads to
\begin{align*}
m = \left\{
\begin{array}{ll}
 1+2c\pi & \text{ for }  c \in (0, \, 1], \\ 
\frac{1}{c}+2\pi  & \text{ for } c \in  (1, \, \infty),\\ 
\end{array} \right. 
\text{ and }
v= \left\{
\begin{array}{ll}
 \sqrt{1+2c\pi} ,  &  \text{ for }  c \in (0, \, 1], \\ 
\sqrt{\frac{1+2c\pi}{c}} &  \text{ for } c \in  (1, \, \infty).\\ 
\end{array} \right.
\end{align*}
Additionally note that exactly this constant $m$ occurred already in  Remark \ref{Bem: LR Zsmhg. krit., subkrit.} concerning the LU decomposition of $\Sigma_n^{\text{cr}}$.
\end{rem}

\section{Stochastic applications} \label{Kap: stoch. appl.}
In this section we consider some of the main results of this manuscript from an stochastic point of view. According to \cite[Theorem 5.2.]{Reitzner} $(\tilde{L}_t^{(\tau_1)}, \ldots ,\tilde{L}_t^{(\tau_n)})$ is asymptotic normal distributed. Recall that $a_i = \tau_i +d$ and abbreviate $a=\tau+d$ as well as $x = \tau+d/2$. With this we can state the following results:

\begin{thm}\label{thm:stochconv-dense}
Assume that $t\delta_t^d \rightarrow \infty$ and $\tau_1, \tau_2 > -d/2$. Then 
$$
D_t^{(\tau_1, \tau_2)} = a_1 \tilde L_t^{(\tau_1)} - a_2 \tilde L_t^{(\tau_2)}
\stackrel{\mathds P}{\to} 0 \text{ as $t \to \infty$.}$$
\end{thm}

\begin{proof}
From the Schur decomposition of $\Sigma_2^{\text{sb}}$ in Corollary \ref{Kor: Diag.matrix, char. Poly. superkrit. Reg.} we gain the rotation matrix \[R = \left( \begin{array}{ccc}   
\frac{a_1}{\sqrt{a_1^2+a_2^2}} & \frac{a_2}{\sqrt{a_1^2+a_2^2}} \\
\frac{-a_2}{\sqrt{a_1^2+a_2^2}} & \frac{a_1}{\sqrt{a_1^2+a_2^2}} \\
\end{array}\right),\] 
such that
$R^{-1}(\tilde{L}_t^{(\tau_1)}, \tilde{L}_t^{(\tau_2)}) ^t\rightarrow N$ with $N \sim \mathcal{N}(0, \, \Sigma)$, where \[\Sigma = \left( \begin{array}{cccc}   
0& 0 \\
0 & V(W)(d\kappa_d)^2(\sum_{i=1}^n 1/a_i)  \\
\end{array}\right).\]

The claim follows from the entries of the first row of $R$.

\end{proof}
In particular, setting $\tau_1=0$, we see that the number of edges $L_t^{(0)}$ asymptotically determines $L_t^{(\tau)}$ for all $\tau$, More precisely, we have
\begin{equation}\label{eq:stochconv-dense}
\tilde L_t^{(\tau)} = d \tilde L_t^{(0)}/a + D_t^{(\tau)} \mbox{ with some }
D_t^{(\tau)} \stackrel{\mathds P}{\to} 0.
\end{equation}

If $ Y \sim \Nc(0, \,\Sigma) $
with
$\Sigma $ a positive definite $2 \times 2$-matrix (i.e. in the critical and subcritical regime), then for the Cholesky factor $G$ with $\Sigma = GG^t $ the equation $G^{-1} \Sigma (G^{-1})^{t} = I_2
$ holds. So \[G^{-1} Y = ( Z_1, Z_2),\] where the $Z_i$ are independent $\Nc (0,1)$-random variables. If $Y= \lim_{t \rightarrow \infty} (\tilde L_t^{(0)}, \tilde L_t^{(\tau)}),$ then $G^{-1} Y = \lim_{t \rightarrow \infty} (b_0 \tilde L_t^{(0)}, c_0 \tilde L_t^{(0)} + c_{\tau} \tilde L_t^{(\tau)})
.$ This leads in particular to:

\begin{thm}\label{thm:stochconv-(sub)crit}
Assume that $t\delta_t^d \rightarrow c \in [0, \infty) $, $\tau > -d/2$, and that $Z\sim \mathcal N(0,1)$ is independent of $\tilde L_t^{(0)}$. Then 
\begin{equation} \label{eq:stochconv-subcritical}
\tilde L_t^{(\tau)} = \frac{d}{a}\tilde L_t^{(0)} 
+ \frac{\tau \sqrt{V(W) d \kappa_d}}{a\sqrt{4x\max\{c,1\}}} Z 
+ D_t^{(\tau)} 
\mbox{ with some }
D_t^{(\tau)} \stackrel{\mathds P}{\to} 0 ,
\end{equation}
as $t \to \infty$.
\end{thm}
\begin{proof}

In the subcritical regime and using Example \ref{Bsp: Cholesky-Zerleg. subkrit. Reg.} we see for the Cholesky factor $G$ of $\Sigma_2$, that
\[G^{-1}=\sqrt{\frac{2}{V(W)d\kappa_d}}\left( \begin{array}{cccc}   
\sqrt{d}& 0 \\
\frac{-d\sqrt{2x}}{\tau}& \frac{a\sqrt{2x}}{\tau} \\
\end{array}\right).\]
Thus, $G^{-1} Y = \sqrt{\frac{2}{V(W)d\kappa_d}}\lim_{t \rightarrow \infty} (\sqrt{d} \tilde L_t^{(0)}, \frac{-d\sqrt{2x}}{\tau} \tilde L_t^{(0)} + \frac{a\sqrt{2x}}{\tau} \tilde L_t^{(\tau)})^t
$, which implies the claim immediately.  In the critical regime we observe from Example \ref{Bsp.Cholesky krit. Reg.} for $c \in (0,1]$ that
\[G^{-1}= \sqrt{\frac{2}{V(W)d\kappa_d}}\left( \begin{array}{cccc}   
\sqrt{\frac{d}{2\kappa_d c+1}}& 0 \\
-\frac{d}{a}\sqrt{\frac{2xa^2}{\tau^2}}& \sqrt{\frac{2xa^2}{\tau^2}}\\
\end{array}\right).\] 
Hence, $G^{-1} Y=\sqrt{\frac{2}{V(W)d\kappa_d}} \lim_{t \rightarrow \infty}  (\sqrt{\frac{d}{2\kappa_d c+1}}\tilde L_t^{(0)}, -\frac{d}{a}\sqrt{\frac{2xa^2}{\tau^2}} \tilde L_t^{(0)} +  \sqrt{\frac{2xa^2}{\tau^2}}\tilde L_t^{(\tau)})^t.$

For $c  \in (1, \infty)$ according to Example \ref{Bsp.Cholesky krit. Reg.} a similar computation leads to the same results with an additional $c$ in the nominator of all entries of $2/V(W)d\kappa_d$. This shows the claim in the critical regime for both cases for $c$.
\end{proof}
It would be of high interest to determine the order of $D_t^{(\tau)}$ 
in Theorems \ref{thm:stochconv-dense} and \ref{thm:stochconv-(sub)crit} as $t \to \infty$, or even its asymptotic 
distribution. This seems to be out of reach at the moment.

\section{Outlook} \label{Kap: Resumee}

Future research activities start by studying open problems from Tables \ref{Tab: first results in all regimes} and \ref{Tab: decompositions in all regimes} for a complete analysis of $\Sigma_n$ with respect to  all algebraic aspects considered before.  For example, though we known in the  critical and subcritical regime the characteristic polynomial, in general, we have no formula for the eigenvalues. A detailed analysis of them would be of interest. For this, a possible strategy is to use matrix decompositions and related methods as discussed in this manuscript. Of course such decompositions are also interesting in their own right and we would like to know more of them like Singular value or QR decompositions. Another example for a research goal is related to the bounds for eigenvalues of $\Sigma_n$ described in Theorem \ref{Satz: Schranken EWe Sigma subkrit. Reg.} and Theorem \ref{Satz: Schranken EWe krit. Reg.}. There are not tight and should be improved.

Stochastic research objectives are related to generalizations
of random geometric graphs. As already observed before, they are a special case of random geometric (simplicial) complexes. Moreover, other Poisson functionals must be considered. In particular, volume power functionals (see \cite{AkinwandeReitzner}) and their covariance matrices will be studied in the future as well as $k$-simplex counting functionals.

\bigskip \medskip

\noindent
\footnotesize {\bf Authors' addresses:}

\smallskip



\noindent 
Institute of Mathematics,
Osnabr\"uck University,
49069 Osnabr\"uck,
Germany\\
{\tt mreitzne@uos.de},\hspace{5 pt}
{\tt troemer@uos.de},\hspace{5 pt}
{\tt mvonwestenho@uos.de}

\end{document}